\documentclass[a4paper,12pt,reqno]{amsart}
\usepackage{amsfonts}
\usepackage{amsmath}
\usepackage{amssymb}
\usepackage[a4paper]{geometry}
\usepackage{mathrsfs}
\usepackage{csquotes}
\usepackage[colorlinks]{hyperref}
\renewcommand\eqref[1]{(\ref{#1})} 
\numberwithin{equation}{section}
\theoremstyle{plain}
\newtheorem{thm}{Theorem}[section]
\newtheorem{prop}[thm]{Proposition}
\newtheorem{cor}[thm]{Corollary}
\newtheorem{lem}[thm]{Lemma}
 
\theoremstyle{definition}

\newtheorem{rem}[thm]{Remark}

\renewcommand{\wp}{\mathfrak S}
\newcommand{\Rn}{\mathbb R^{n}}

\begin{document}
   \title[Hardy inequalities for magnetic operators]
   {Hardy inequalities for Landau Hamiltonian and for Baouendi-Grushin operator with Aharonov-Bohm type magnetic field. Part I.}

\author[A. Laptev]{Ari Laptev}
\address{
  Ari Laptev:
 \endgraf
  Department of Mathematics
  \endgraf
  Imperial College London
  \endgraf
  180 Queen's Gate, London SW7 2AZ
  \endgraf
  United Kingdom
  \endgraf
  and
  \endgraf
  St. Petersburg State University
  \endgraf
  St. Petersburg, Russia
  \endgraf
  {\it E-mail address} {\rm a.laptev@imperial.ac.uk}
  }

\author[M. Ruzhansky]{Michael Ruzhansky}
\address{
  Michael Ruzhansky:
  \endgraf
  Department of Mathematics
  \endgraf
  Imperial College London
  \endgraf
  180 Queen's Gate, London SW7 2AZ
  \endgraf
  United Kingdom
  \endgraf
  {\it E-mail address} {\rm m.ruzhansky@imperial.ac.uk}
  }

\author[N. Yessirkegenov]{Nurgissa Yessirkegenov}
\address{
  Nurgissa Yessirkegenov:
  \endgraf
  Institute of Mathematics and Mathematical Modelling
  \endgraf
  125 Pushkin str.
  \endgraf
  050010 Almaty
  \endgraf
  Kazakhstan
  \endgraf
  and
  \endgraf
  Department of Mathematics
  \endgraf
  Imperial College London
  \endgraf
  180 Queen's Gate, London SW7 2AZ
  \endgraf
  United Kingdom
  \endgraf
  {\it E-mail address} {\rm n.yessirkegenov15@imperial.ac.uk}
  }

\thanks{AL was supported by the RSF Grant No. 15-11-30007. MR was supported in parts by the EPSRC
 grant EP/K039407/1 and by the Leverhulme Grant RPG-2014-02. NY was supported by the MESRK grant 0825/GF4. No new data was collected or
generated during the course of research.}


     \begin{abstract}
In this paper we prove the Hardy inequalities for the quadratic form of the Laplacian with the Landau Hamiltonian magnetic field. Moreover, we obtain Poincar\'e type inequality and inequalities with more general families of weights, all with estimates for the remainder terms of these inequalities.
Furthermore, we establish
weighted Hardy inequalities for the quadratic form of the magnetic Baouendi-Grushin operator for the magnetic field of Aharonov-Bohm type. For these, we show refinements of the known Hardy inequalities for the Baouendi-Grushin operator involving radial derivatives in some of the variables.
The corresponding uncertainty type principles are also obtained.
     \end{abstract}
     \maketitle

\section{Introduction}

The purpose of this paper is to prove the weighted Hardy inequality for the quadratic form of the Landau Hamiltonian and for the magnetic Baouendi-Grushin operator with Aharonov-Bohm type magnetic field. In Part II of this paper we investigate and present the corresponding Caffarelli-Kohn-Nirenberg inequalities for the Landau Hamiltonian and for the Baouendi-Grushin operator, with and without magnetic fields.

The classical Hardy inequality for functions $f\in C_{0}^{\infty}(\Rn\backslash0)$ is
\begin{equation}\label{Intro_1}
\int_{\Rn}|\nabla f(w)|^{2}dw\geq \left(\frac{n-2}{2}\right)^{2}\int_{\Rn}\frac{|f(w)|^{2}}{|w|^{2}}dw, \quad n\geq3,
\end{equation}
where the constant $\left(\frac{n-2}{2}\right)^{2}$ is sharp but not attained. There exists a large literature concerning different versions of Hardy's inequalities and their applications. However, since we are interested in the inequalities associated with the Landau Hamiltonian and with the Baouendi-Grushin operator, let us only recall known results in these directions.

\subsection{Baouendi-Grushin operator}
\label{SEC:iGR}

The Hardy inequality \eqref{Intro_1} has been generalised for Baouendi-Grushin vector fields by Garofalo \cite{G},
\begin{equation}\label{Intro_2}
\int_{\Rn}(|\nabla_{x}f|^{2}+|x|^{2\gamma}|\nabla_{y}f|^{2})dxdy\geq \left(\frac{Q-2}{2}\right)^{2}
\int_{\Rn}\left(\frac{|x|^{2\gamma}}{|x|^{2+2\gamma}+(1+\gamma)^{2}|y|^{2}}\right)|f|^{2}dxdy,
\end{equation}
where $x\in\mathbb{R}^{m}$, $y\in\mathbb{R}^{k}$ with $n=m+k$, $m,k\geq1$, $\gamma\geq0$, $Q=m+(1+\gamma)k$ and
$f\in C_{0}^{\infty}(\mathbb{R}^{m}\times\mathbb{R}^{k}\backslash\{(0,0)\})$. Here, $\nabla_{x}f$ and $\nabla_{y}f$ are the gradients of $f$ in the variables $x$ and $y$, respectively. The inequality \eqref{Intro_2} recovers \eqref{Intro_1} when $\gamma=0$.

Let us put this result in perspective.
Let $z=(x_{1},...,x_{m}, y_{1},...,y_{k})=(x,y)\in \mathbb{R}^{m}\times\mathbb{R}^{k}$ with $k,m\geq1$, $k+m=n$ and $\gamma\geq0$. Let us consider the vector fields
$$X_{i}=\frac{\partial}{\partial x_{i}}, \;i=1,...,m, \;\;\; Y_{j}=|x|^{\gamma}\frac{\partial}{\partial y_{j}}, \;j=1,...,k.$$
The corresponding sub-elliptic gradient, which is the $n$ dimensional vector field, is then defined as
\begin{equation}\label{subgrad}
\nabla_{\gamma}:=(X_{1},...,X_{m}, Y_{1},...,Y_{k})=(\nabla_{x}, |x|^{\gamma}\nabla_{y}).
\end{equation}
The Baouendi-Grushin operator on $\mathbb{R}^{m+k}$ is defined by
\begin{equation}\label{Grush_op}
\triangle_{\gamma}=\sum_{i=1}^{m}X_{i}^{2}+\sum_{j=1}^{k}Y_{j}^{2}=\triangle_{x}+|x|^{2\gamma}\triangle_{y}=\nabla_{\gamma}\cdot \nabla_{\gamma},
\end{equation}
where $\triangle_{x}$ and $\triangle_{y}$ are the Laplace operators in the variables $x\in \mathbb{R}^{m}$ and $y\in \mathbb{R}^{k}$, respectively.  The Baouendi-Grushin operator for an even positive integer $\gamma$ is a sum of squares of $C^{\infty}$ vector fields satisfying
H\"{o}rmander condition
$${\rm rank} \;{\rm Lie} [X_{1},...,X_{m}, Y_{1},...,Y_{k}]=n.$$
We can define on $\mathbb{R}^{m+k}$ the anisotropic dilation attached to $\triangle_{\gamma}$ as
$$\delta_{\lambda}(x,y)=(\lambda x, \lambda^{1+\gamma} y)$$
for $\lambda>0$, and the homogeneous dimension with respect to this dilation is
\begin{equation}\label{hom_dim}
Q=m+(1+\gamma)k.
\end{equation}
A change of variables formula for the Lebesgue measure implies that
$$d\circ\delta_{\lambda}(x,y)=\lambda^{Q}dxdy.$$
It is easy to check that
$$X_{i}(\delta_{\lambda})=\lambda\delta_{\lambda}(X_{i}),\quad Y_{i}(\delta_{\lambda})=\lambda\delta_{\lambda}(Y_{i}),$$
and hence
$$\nabla_{\gamma}\circ \delta_{\lambda}=\lambda\delta_{\lambda} \nabla_{\gamma}.$$
Let $\rho(z)$ be the corresponding distance function from the origin for $z=(x,y)\in \mathbb{R}^{m}\times \mathbb{R}^{k}$:
\begin{equation}\label{dist}
\rho=\rho(z):=(|x|^{2(1+\gamma)}+(1+\gamma)^{2}|y|^{2})^{\frac{1}{2(1+\gamma)}}.
\end{equation}
By a direct calculation one obtains
\begin{equation}\label{formula}|\nabla_{\gamma}\rho|=\frac{|x|^{\gamma}}{\rho^{\gamma}}.
\end{equation}
The described setup may be thought of as a special case of the setting of homogeneous groups, see e.g. \cite{FR}.

The weighted $L^{p}$-versions of \eqref{Intro_2} have been obtained  by D'Ambrosio \cite{A2}:
Let $\Omega\subset\Rn$ be an open set. Let $p>1$, $k,m\geq1$, $\alpha, \beta\in\mathbb{R}$ be such that $m+(1+\gamma)k>\alpha-\beta$ and $m>\gamma p-\beta$. Then for every $f\in D_{\gamma}^{1,p}(\Omega, |x|^{\beta-\gamma p}
\rho^{(1+\gamma)p-\alpha})$ we have
\begin{equation}\label{Amb1}
\int_{\Omega}|\nabla_{\gamma}f|^{p}|x|^{\beta-\gamma p}\rho^{(1+\gamma)p-\alpha}dxdy\geq\left(\frac{Q+\beta-\alpha}{p}\right)^{p}\int_{\Omega}|f|^{p}\frac{|x|^{\beta}}{\rho^{\alpha}}dxdy,
\end{equation}
where $D_{\gamma}^{1,p}(\Omega,\omega)$ denotes the closure of $C_{0}^{\infty}(\Omega)$ in the norm $\left(\int_{\Omega}|\nabla_{\gamma}f|^{p}\omega dzdy\right)^{1/p}$ for $\omega\in L^{1}_{loc}(\Omega)$ with $\omega>0$ a.e. on $\Omega$.

If $0\in\Omega$, then the constant $\left(\frac{Q+\beta-\alpha}{p}\right)^{p}$ in \eqref{Amb1} is sharp.
The inequality \eqref{Amb1} has also been established in \cite{Kombe15}, and in \cite{ShY12} for $\Omega=\Rn$ with sharp constant. Moreover, in \cite{ShY12}, a Hardy-Rellich type inequality for the Baouendi-Grushin operator is obtained in $L^{2}$ with sharp constant:
$$\left(\frac{Q-\alpha-2}{2}\right)^{2}\int_{\Rn}|\nabla_{\gamma}f|^{2}
\rho^{\alpha}\leq\int_{\Rn}|\triangle_{\gamma}f|^{2}\rho^{\alpha+2}|\nabla_{\gamma}\rho|^{-2},$$
where $p>1$, $\frac{2-Q}{3}\leq \alpha \leq Q-2$, $f\in C^{\infty}_{0}(\Rn\backslash\{0\})$.

The inequalities of this type have been also studied for some sub-elliptic operators of different types (see e.g. \cite{G}, \cite{GL}, \cite{A1}, \cite{A2}, \cite{NCH}, \cite{K} and \cite{DGN}).
For Hardy and Caffarelli-Kohn-Nirenberg inequalities on more general homogeneous Carnot groups and the literature review including the Heisenberg group we refer to
\cite{Ruzhansky-Suragan:JDE,Ruzhansky-Suragan:Layers}, for the anisotropic versions of the usual $L^2$ and $L^p$ Cafarelli-Kohn-Nirenberg inequalities we refer to \cite{Ruzhansky-Suragan:L2-CKN} and \cite{ORS16}, respectively, and for Hardy inequalities for more general sums of squares of vector fields we refer to \cite{Ruzhansky-Suragan:squares}.

Here we obtain the following refinement of Hardy inequalities for the Baouendi-Grushin operator:

\begin{itemize}
\item ({\bf Weighted refined Hardy inequalities for Grushin operators}) Let $(x,y)=(x_{1},...,x_{m}, y_{1},...,y_{k})\in \mathbb{R}^{m}\times\mathbb{R}^{k}$ with $k,m\geq1$, $k+m=n$. Let $Q+\alpha_{1}-2>0$ and $m+\gamma \alpha_{2}>0$. Then we have the following Hardy type inequality for all complex-valued functions $f\in C_{0}^{\infty}(\Rn\backslash\{0\})$
\begin{multline}\label{EQ:rad}
\qquad \quad \int_{\Rn}\rho^{\alpha_{1}}|\nabla_{\gamma}\rho|^{\alpha_{2}}\left(
\left|\frac{d}{d|x|}f\right|^{2}+|x|^{2\gamma}|\nabla_{y}f|^{2}\right)dxdy \\
\geq \left(\frac{Q+\alpha_{1}-2}{2}\right)^{2}
\int_{\Rn}\rho^{\alpha_{1}}|\nabla_{\gamma}\rho|^{\alpha_{2}}\frac{|\nabla_{\gamma}\rho|^{2}}{\rho^{2}}|f|^{2}dxdy,
\end{multline}
where the constant $\left(\frac{Q+\alpha_{1}-2}{2}\right)^{2}$ is sharp.
\end{itemize}
Already in the absence of weights, i.e. for $\alpha_1=\alpha_2=0$, the estimate \eqref{EQ:rad} is new. Moreover, it is also new as a family of inequalities in the case $\gamma=0$: for $k=0$ it reduces to the classical Hardy inequality \eqref{Intro_1}, while for $m=0$ it reduces to the radial version established in  \cite{IIO:Lp-Hardy}, see also \cite{MOW:Hardy-Hayashi} (always for $\gamma=0$). We note that since we can estimate
$\left|\frac{d}{d|x|}f\right|\leq |\nabla_x f|$, inequality \eqref{EQ:rad} also gives a refinement to the inequality \eqref{Amb1} for $p=2$.

The estimate \eqref{EQ:rad}, in addition to its own interest, will play an important role in the derivation of estimates for magnetic operators.

\subsection{Magnetic Baouendi-Grushin operator}
\label{SEC:iGRm}

In \cite{LW} and \cite{AL11}, Hardy inequalities for some magnetic forms were obtained. For example, in \cite{AL11} for the quadratic form of the following magnetic Grushin operator
$$G_{\mathcal{A}}=-(\nabla_{G}+i\beta\mathcal{A}_0)^{2},$$
the following Hardy inequality for $-\frac{1}{2}\leq \beta \leq \frac{1}{2}$ was proved:
\begin{equation}\label{EQ:AL}
 \int_{\mathbb{R}^{3}}|(\nabla_{G}+i\beta\mathcal{A}_0)f|^{2}dzdt\geq (1+\beta^{2})\int_{\mathbb{R}^{3}}\frac{|z|^{2}}{d^{4}}|f|^{2}dzdt,
\end{equation}
where
$$\mathcal{A}_0=(\mathcal{A}_{1},\mathcal{A}_{2},\mathcal{A}_{3},\mathcal{A}_{4})=\left(-\frac{\partial_{y}d}{d},
\frac{\partial_{x}d}{d}, -2y\frac{\partial_{t}d}{d}, 2x \frac{\partial_{t}d}{d}\right),$$
$\nabla_{G}=(\partial_{x}, \partial_{y}, 2x\partial_{t}, 2y\partial_{t})$ with $z=(x,y)$, $|z|=\sqrt{x^{2}+y^{2}}$, $\beta\in\mathbb{R}$ is a \enquote{flux} and
$d(z,t)=(|z|^{4}+t^{2})^{\frac{1}{4}}$ is the Kaplan distance.

The following results extend the estimate \eqref{EQ:AL}.
While the most physical setting is $y\in \mathbb R^1$, we can obtain the results for any $y\in\mathbb R^k$.

\begin{itemize}
\item ({\bf Hardy inequality for the magnetic Baouendi-Grushin operator})
Let $(x,y)=(x_{1}, x_{2}, y)\in \mathbb{R}^{2}\times\mathbb{R}^{k}$. Let $\alpha_{1},\alpha_{2}, \beta\in\mathbb{R}$ be such that $\alpha_{1}+k(\gamma+1)>0$ and $\alpha_{2}\gamma+2>0$.
Let us define the
Aharonov-Bohm type magnetic field
\begin{equation}\label{another_mag2}
\mathcal{\widetilde{A}}:=
\left(-\frac{\partial_{x_{2}} \rho}{\rho}, \frac{\partial_{x_{1}} \rho}{\rho},-\frac{|x|^{\gamma}}{\sqrt{2}}
\frac{\nabla_{y} \rho}{\rho}, \frac{|x|^{\gamma}}{\sqrt{2}}
\frac{\nabla_{y} \rho}{\rho}\right),
\end{equation}
and the corresponding gradient
\begin{equation}\label{another_mag3}
\widetilde{\nabla_{\gamma}}=\left(\frac{\partial}{\partial x_{1}}, \frac{\partial}{\partial x_{2}}, \frac{|x|^{\gamma}}{\sqrt{2}}\nabla_{y}, \frac{|x|^{\gamma}}{\sqrt{2}}\nabla_{y}\right).
\end{equation}
Then for any complex-valued function $f\in C^{\infty}_{0}(\mathbb{R}^{2+k}\backslash\{0\})$ we have the following weighted Hardy inequality for the magnetic Baouendi-Grushin operator
$$\int_{\mathbb{R}^{2+k}}\rho^{\alpha_{1}}|\widetilde{\nabla_{\gamma}}\rho|^{\alpha_{2}}|(\widetilde{\nabla_{\gamma}}+i\beta\mathcal{\widetilde{A}})f|
^{2}dxdy
\geq\left(\left(\frac{\alpha_{1}+k(\gamma+1)}{2}\right)^{2}+\beta^{2}\right)$$
$$
\times\int_{\mathbb{R}^{2+k}}\rho^{\alpha_{1}}|\widetilde{\nabla_{\gamma}}\rho|
^{\alpha_{2}}\frac{|x|^{2\gamma}}{\rho^{2\gamma+2}}|f|^{2}dxdy,
$$
where the constant $\left(\left(\frac{\alpha_{1}+k(\gamma+1)}{2}\right)^{2}+\beta^{2}\right)$ is sharp (so that the constant in \eqref{EQ:AL} is actually also sharp).
Moreover, we have the estimate for the remainder term in this inequality:
$$\int_{\mathbb{R}^{2+k}}\rho^{\alpha_{1}}|\widetilde{\nabla_{\gamma}}\rho|^{\alpha_{2}}|(\widetilde{\nabla_{\gamma}}+i\beta\mathcal{\widetilde{A}})f|
^{2}dxdy
\geq\left(\left(\frac{\alpha_{1}+k(\gamma+1)}{2}\right)^{2}+\beta^{2}\right)$$
$$
\times\int_{\mathbb{R}^{2+k}}\rho^{\alpha_{1}}|\widetilde{\nabla_{\gamma}}\rho|
^{\alpha_{2}}\frac{|x|^{2\gamma}}{\rho^{2\gamma+2}}|f|^{2}dxdy
$$
$$+\int_{\mathbb{R}^{2+k}}\rho^{\alpha_{1}}|\widetilde{\nabla_{\gamma}}\rho|
^{\alpha_{2}}\frac{|f|^{2}-|f_{0}(|x|,y)|^{2}}{|x|^{2}}dxdy,
$$
where $f_{0}(|x|,y)=\frac{1}{2\pi}\int_{0}^{2\pi}f(|x|,\phi,y)d\phi$ with $(|x|,\phi)$ being the polar decomposition of $x$, so that the last term in the above inequality is nonnegative.
\item
({\bf Uncertainty type principle}) Let $(x,y)=(x_{1}, x_{2}, y)\in \mathbb{R}^{2}\times\mathbb{R}^{k}$. Let $\alpha_{1}, \alpha_{2}, \beta\in\mathbb{R}$ be such that $\alpha_{1}+k(\gamma+1)>0$ and $\alpha_{2}\gamma+2>0$. Then for any complex-valued function $f\in C^{\infty}_{0}(\mathbb{R}^{2+k}\backslash\{0\})$ we have
$$\|\rho^{\frac{\alpha_{1}}{2}}|\widetilde{\nabla_{\gamma}}\rho|^{\frac{\alpha_{2}}{2}}(\widetilde{\nabla_{\gamma}}+
i\beta\mathcal{\widetilde{A}})f\|_{L^{2}(\mathbb{R}^{2+k})}
\|f\|_{L^{2}(\mathbb{R}^{2+k})}$$
$$\geq\left(\left(\frac{\alpha_{1}+k(\gamma+1)}{2}\right)^{2}+\beta^{2}\right)^{\frac{1}{2}}
\int_{\mathbb{R}^{2+k}}\rho^{\frac{\alpha_{1}}{2}}|\widetilde{\nabla_{\gamma}}\rho|^{\frac{\alpha_{2}}{2}}
\frac{|x|^{\gamma}}{\rho^{\gamma+1}}|f|^{2}dxdy.
$$
\item({\bf Magnetic Baouendi-Grushin operator with constant magnetic field})
In Remark \ref{Grushin_constant_mag}, we also give inequalities for the magnetic Baouendi-Grushin operator on $\mathbb{C}^{n}$ with the constant magnetic field
$$\mathcal{L}_{G}=\sum_{j=1}^{n}((i\partial_{x_{j}}+\psi_{1,j}(y_{j}))^{2}+(i|x|^{\gamma}\partial_{y_{j}}+\psi_{2,j}(x_{j}))^{2}).$$
\end{itemize}

\subsection{Landau Hamiltonian}
\label{SEC:introLH}

Let us recall that the Landau Hamiltonian (or the twisted Laplacian) on $\mathbb{C}^{n}$ is defined as
\begin{equation}\label{twist1}
\mathcal{L}=\sum_{j=1}^{n}\left[\left(i\partial_{x_{j}}+\frac{y_{j}}{2}\right)^{2}+
\left(i\partial_{y_{j}}-\frac{x_{j}}{2}\right)^{2}\right].
\end{equation}
Setting
$$\widetilde{X}_{j}=\partial_{x_{j}}-\frac{1}{2}iy_{j}\;\textrm{ and }\; \widetilde{Y}_{j}=\partial_{y_{j}}+\frac{1}{2}ix_{j},$$
we have
\begin{equation}\label{twist7}
\mathcal{L}=-\sum_{j=1}^{n}(\widetilde{X}_{j}^{2}+\widetilde{Y}_{j}^{2}).
\end{equation}
The twisted Laplacian can be also written as $\mathcal{L}=-\triangle+\frac{1}{4}(|x|^{2}+|y|^{2})+iN$, where
\begin{equation}\label{N}
N=\sum_{j=1}^{n}(y_{j}\partial_{x_{j}}-x_{j}\partial{y_{j}})
\end{equation}
is the rotation field.
Let $\nabla_{\mathcal{L}}$ be the gradient operator associated with $\mathcal{L}$:
\begin{equation}\label{twist2}
\nabla_{\mathcal{L}}f=(\widetilde{X}_{1}f,...,\widetilde{X}_{n}f, \widetilde{Y}_{1}f,...,\widetilde{Y}_{n}f).
\end{equation}
Let $W_{\mathcal{L}}^{1,2}(\mathbb{C}^{n})$ be the Sobolev space defined by
\begin{equation}\label{twist_def}W_{\mathcal{L}}^{1,2}(\mathbb{C}^{n})=\{f\in L^{2}(\mathbb{C}^{n}):\widetilde{X}_{j}f, \widetilde{Y}_{j}f\in L^{2}(\mathbb{C}^{n}), 1\leq j\leq n\}.
\end{equation}

In the recent paper \cite{ARS17}, a version of the Hardy inequality for the twisted Laplacian with Landau Hamiltonian magnetic field was established for {\em real-valued} functions $f\in W_{\mathcal{L}}^{1,2}(\mathbb{C}^{n})$, namely, the inequality
$$
\frac{1}{4}\int_{\mathbb{C}^{n}}|f|^{2}\omega(z)dz\leq\int_{\mathbb{C}^{n}}|\nabla_{\mathcal{L}}f|^{2}dz,
$$
with the weight
$$
\omega(z)=\left(\frac{|\nabla_{\mathcal{L}}\mathbb{E}|^{2}}{\mathbb{E}^{2}}+\frac{|z|^{2}}{4}\right),
$$
where $\mathbb{E}$ is a fundamental solution to the twisted Laplacian on $\mathbb{C}^{n}$.

In this paper we obtain the versions of Hardy inequalities for the Landau Hamiltonian for both {\em complex-valued} and {\em real-valued functions}, as well as estimates for the remainders.

In fact, we obtain also results for more general operators $\widetilde{\mathcal{L}}_\psi$
of the form
$$\widetilde{\mathcal{L}}_\psi=\sum_{j=1}^{n}\left[\left(i\partial_{x_{j}}+\psi(|z|)y_{j}\right)^{2}+
\left(i\partial_{y_{j}}-\psi(|z|)x_{j}\right)^{2}\right],$$
where $\psi(|z|)$ is a radial real-valued differentiable function, $z=(x,y)$.
Setting
$$\check{{X}_{j}}=\partial_{x_{j}}-i\psi(|z|)y_{j}\;\textrm{ and }\; \check{Y}_{j}=\partial_{y_{j}}+i\psi(|z|)x_{j},$$
we write
\begin{equation}\label{twist3_1_2}
\widetilde{\nabla_{\mathcal{L}_\psi}}f=(\check{X}_{1}f, \ldots, \check{X}_{n}f, \check{Y}_{1}f, \ldots, \check{Y}_{n}f).
\end{equation}
We note that for $\psi(|z|)=\frac12$ we recover the classical Landau Hamiltonian, i.e.
$$ \widetilde{\mathcal{L}}_{1/2}=-\mathcal{L} \textrm{ and }
\widetilde{\nabla_{\mathcal{L}_{1/2}}}=\nabla_{\mathcal{L}}.$$
As usual, we will identify $\mathbb{C}\cong\mathbb{R}^{2}$.
\begin{itemize}
\item ({\bf Hardy inequalities for the Landau-Hamiltonian $\widetilde{\mathcal{L}}_\psi$})
Let $\psi=\psi(|z|)$ be a radial real-valued function such that $\psi\in L^{2}_{loc}(\mathbb{C}\backslash\{0\})$.
For a function $f$ we will  use the notation $f_{0}(|z|):=\frac{1}{2\pi}\int_{0}^{2\pi}f(|z|,\phi)d\phi$.
Then we have the following inequalities:

\medskip
(i) {\bf (Hardy-Sobolev inequality)}  For any $\theta_{1}\in\mathbb{R}\backslash\{0\}$ we have
$$\int_{\mathbb{C}}\frac{|\widetilde{\nabla_{\mathcal{L}}}f|^{2}}{|z|^{2\theta_{1}}}dz-
\theta_{1}^{2}\int_{\mathbb{C}}\frac{|f|^{2}}{|z|^{2\theta_{1}+2}}dz\geq
\int_{\mathbb{C}}\frac{(\psi(|z|))^{2}}{|z|^{2\theta_{1}-2}}|f|^{2}dz$$
\begin{equation}\label{EQ:lh1}
+\int_{\mathbb{C}}\frac{|f|^{2}-|f_{0}(|z|)|^{2}}{|z|^{2\theta_{1}+2}}dz,
\end{equation}
for all complex-valued functions $f\in C_{0}^{\infty}(\mathbb{R}^{2}\backslash\{0\})$.

\medskip
(ii) {\bf (Logarithmic Hardy inequality)} We have
$$\int_{\mathbb{C}}|\widetilde{\nabla_{\mathcal{L}}}f|^{2}|\log|z||^{2}dz-\frac{1}{4}\int_{\mathbb{C}}|f|^{2}dz\geq
\int_{\mathbb{C}}(\psi(|z|))^{2}|z|^{2}|\log|z||^{2}|f|^{2}dz$$
\begin{equation}\label{EQ:lh2}
+\int_{\mathbb{C}}\frac{|f|^{2}-|f_{0}(|z|)|^{2}}{|z|^{2}}|\log|z||^{2}dz,
\end{equation}
for all complex-valued functions $f\in C_{0}^{\infty}(\mathbb{R}^{2}\backslash\{0\})$.

\medskip
(iii) {\bf (Poincar\'e inequality)}
Let $\Omega$ be a bounded domain in $\mathbb{C}$ and $R=\underset{z\in\Omega}{\rm sup}\{|z|\}$. Then we have
$$\int_{\Omega}|\widetilde{\nabla_{\mathcal{L}}}f|^{2}dz-\frac{1}{R^{2}}\int_{\Omega}|f|^{2}dz\geq
\int_{\Omega}(\psi(|z|))^{2}|z|^{2}|f|^{2}dz$$
\begin{equation}\label{EQ:lh3}
+\int_{\Omega}\frac{|f|^{2}-|f_{0}(|z|)|^{2}}{|z|^{2}}dz,
\end{equation}
for all complex-valued functions $f\in \widehat{\mathfrak{L}}_{0}^{1,2}(\Omega)$ satisfying $\frac{d}{d|z|}f\in L^{2}(\Omega)$, where the space $\widehat{\mathfrak{L}}_{0}^{1,2}(\Omega)$ is defined in \eqref{space}.

\medskip
(iv) {\bf (Hardy-Sobolev inequality with superweights)}
Let $\theta_{2}, \theta_{3}, \theta_{4}$, $a$, $b\in\mathbb{R}$ with $a,b>0$, $\theta_{2} \theta_{3}<0$ and $2\theta_{4}\leq \theta_{2}\theta_{3}$.
Then we have
$$\int_{\mathbb{C}}\frac{(a+b|z|^{\theta_{2}})^{\theta_{3}}}{|z|^{2\theta_{4}}}|\widetilde{\nabla_{\mathcal{L}}}f|^{2}dz\geq
\frac{\theta_{2}\theta_{3}-2\theta_{4}}{2}\int_{\mathbb{C}}\frac{(a+b|z|^{\theta_{2}})^{\theta_{3}}}{|z|^{2\theta_{4}+2}}|f|^{2}dz$$
\begin{equation}\label{EQ:lh4}
\qquad +\int_{\mathbb{C}}\frac{(\psi(|z|))^{2}(a+b|z|^{\theta_{2}})^{\theta_{3}}}{|z|^{2\theta_{4}-2}}|f|^{2}dz
+\int_{\mathbb{C}}\frac{(a+b|z|^{\theta_{2}})^{\theta_{3}}(|f|^{2}-|f_{0}(|z|)|^{2})}{|z|^{2\theta_{4}+2}}dz,
\end{equation}
for all complex-valued functions $f\in C_{0}^{\infty}(\mathbb{R}^{2}\backslash\{0\})$.
Weights of this type has appeared in \cite{GM11} as well as in \cite{RSY17}, and are called the superweights due to the freedom in the choice of indices.

\medskip
{\em All terms on the right hand sides of inequalities \eqref{EQ:lh1}, \eqref{EQ:lh2}, \eqref{EQ:lh3} and \eqref{EQ:lh4} are non-negative}, therefore, they give the remainder estimates as well as Hardy's inequalities:
\begin{equation}\label{EQ:lh11}
\int_{\mathbb{C}}\frac{|\widetilde{\nabla_{\mathcal{L}_\psi}}f|^{2}}{|z|^{2\theta_{1}}}dz\geq
\theta_{1}^{2}\int_{\mathbb{C}}\frac{|f|^{2}}{|z|^{2\theta_{1}+2}}dz+
\int_{\mathbb{C}}\frac{(\psi(|z|))^2}{|z|^{2\theta_{1}-2}}|f|^{2}dz,
\end{equation}
\begin{equation}\label{EQ:lh22}
\int_{\mathbb{C}}|\widetilde{\nabla_{\mathcal{L}_\psi}}f|^{2}|\log|z||^{2}dz\geq \frac{1}{4}\int_{\mathbb{C}}|f|^{2}dz+
\int_{\mathbb{C}}(\psi(|z|))^2|z|^{2}|\log|z||^{2}|f|^{2}dz,
\end{equation}
\begin{equation}\label{EQ:lh33}
\int_{\Omega}|\widetilde{\nabla_{\mathcal{L}}}f|^{2}dz\geq\frac{1}{R^{2}}\int_{\Omega}|f|^{2}dz+
\int_{\Omega}(\psi(|z|))^{2}|z|^{2}|f|^{2}dz
\end{equation}
and
$$\int_{\mathbb{C}}\frac{(a+b|z|^{\theta_{2}})^{\theta_{3}}}{|z|^{2\theta_{4}}}|\widetilde{\nabla_{\mathcal{L}}}f|^{2}dz\geq
\frac{\theta_{2}\theta_{3}-2\theta_{4}}{2}\int_{\mathbb{C}}\frac{(a+b|z|^{\theta_{2}})^{\theta_{3}}}{|z|^{2\theta_{4}+2}}|f|^{2}dz$$
\begin{equation}\label{EQ:lh44}
+\int_{\mathbb{C}}\frac{(\psi(|z|))^{2}(a+b|z|^{\theta_{2}})^{\theta_{3}}}{|z|^{2\theta_{4}-2}}|f|^{2}dz.
\end{equation}
For $\psi(|z|)=\frac12$ this yields Hardy inequalities and remainder estimates for the classical Landau Hamiltonian operator.
\end{itemize}

The Hardy inequalities for a magnetic Baouendi-Grushin operator with Aharonov-Bohm type magnetic field are proved in Section \ref{Sec3}. In Section \ref{Sec4} we prove the Hardy inequalities for the twisted Laplacian with the Landau-Hamiltonian type magnetic field.

\section{Weighted Hardy inequalities for magnetic Baouendi-Grushin operator with Aharonov-Bohm type magnetic field} \label{Sec3}

In this section we establish weighted Hardy inequalities for the quadratic form of the magnetic Baouendi-Grushin operator with Aharonov-Bohm type magnetic field.
We adapt all the notation introduced in Sections \ref{SEC:iGR} and \ref{SEC:iGRm}, namely,
$\nabla_{\gamma}$, $\rho$ and $Q$ defined in \eqref{subgrad}, \eqref{dist} and \eqref{hom_dim}, respectively. Recalling these for conveniece of the reader, we have
\begin{equation}\label{subgrad2}
\nabla_{\gamma}=(\nabla_{x}, |x|^{\gamma}\nabla_{y}),\quad
(x,y)\in \mathbb{R}^{m}\times\mathbb{R}^{k},\quad
Q=m+(1+\gamma)k,\; \gamma\geq 0,
\end{equation}
and the magnetic filed
 $\mathcal{A}$ is defined here as
\begin{equation}\label{mag_fieldi2}
\mathcal{A}=\frac{\nabla_{\gamma}\rho}{\rho}=
\left(\frac{\nabla_{x}\rho}{\rho},|x|^{\gamma}\frac{\nabla_{y}\rho}{\rho}\right)\in\mathbb{R}^{m}\times\mathbb{R}^{k}.
\end{equation}

We start with a simple inequality showing the best constants one can expect.

\begin{lem}\label{lem1}
Let $\Omega\subset \Rn$ be an open set. Let $(x,y)=(x_{1},...,x_{m}, y_{1},...,y_{k})\in \mathbb{R}^{m}\times\mathbb{R}^{k}$ with $k,m\geq1$, $k+m=n$. Let $\alpha_{1}, \alpha_{2}, \beta\in\mathbb{R}$ be such that
$$Q+\alpha_{1}-2>0 \textrm{ and } m+\alpha_{2}\gamma>0.$$
Then for any real-valued function $f\in C^{\infty}_{0}(\Omega)$ we have the following weighted Hardy inequality for the magnetic Baouendi-Grushin operator
$$\int_{\Omega}\rho^{\alpha_{1}}|\nabla_{\gamma}\rho|^{\alpha_{2}}|(\nabla_{\gamma}+i\beta\mathcal{A})f|^{2}dxdy
\geq\left(\left(\frac{Q+\alpha_{1}-2}{2}\right)^{2}+\beta^{2}\right)$$
\begin{equation}\label{Hardy1}
\times\int_{\Omega}\rho^{\alpha_{1}}|\nabla_{\gamma}\rho|^{\alpha_{2}}\frac{|x|^{2\gamma}}{\rho^{2\gamma+2}}|f|^{2}dxdy,
\end{equation}
Moreover, if $0\in\Omega$, then the constant in \eqref{Hardy1} is sharp.
\end{lem}
\begin{proof}[Proof of Lemma \ref{lem1}] By opening brackets we have
$$\int_{\Omega}\rho^{\alpha_{1}}|\nabla_{\gamma}\rho|^{\alpha_{2}}|(\nabla_{\gamma}+i\beta\mathcal{A})f|^{2}dxdy=
\int_{\Omega}\rho^{\alpha_{1}}|\nabla_{\gamma}\rho|^{\alpha_{2}}|\nabla_{\gamma}f|^{2}dxdy$$
\begin{equation}\label{Hardy2}
+\beta^{2}\int_{\Omega}\rho^{\alpha_{1}}|\nabla_{\gamma}\rho|^{\alpha_{2}}|\mathcal{A}f|^{2}dxdy.
\end{equation}
Taking into account \eqref{formula} and putting $p=2$, $\alpha=\gamma(\alpha_{2}+2)+2-\alpha_{1}$, $\beta=\gamma(\alpha_{2}+2)$ in \eqref{Amb1}, we have for any $Q+\alpha_{1}-2>0$ and $m+\alpha_{2}\gamma>0$ that
\begin{equation}\label{Hardy3}\int_{\Omega}\rho^{\alpha_{1}}|\nabla_{\gamma}\rho|^{\alpha_{2}}|\nabla_{\gamma}f|^{2}dxdy\geq \left(\frac{Q+\alpha_{1}-2}{2}\right)^{2}
\int_{\Omega}\rho^{\alpha_{1}}|\nabla_{\gamma}\rho|^{\alpha_{2}}\frac{|\nabla_{\gamma}\rho|^{2}}{\rho^{2}}|f|^{2}dxdy.
\end{equation}
Taking into account the form of the magnetic field $\mathcal{A}=(\mathcal{A}_{1}, \mathcal{A}_{2})=\left(\frac{\nabla_{x}\rho}{\rho},|x|^{\gamma}\frac{\nabla_{y}\rho}{\rho}\right)$ in \eqref{mag_fieldi2}, and by a direct calculation one finds
$$\beta^{2}\int_{\Omega}\rho^{\alpha_{1}}|\nabla_{\gamma}\rho|^{\alpha_{2}}|\mathcal{A}f|^{2}dxdy=
\beta^{2}\int_{\Omega}\rho^{\alpha_{1}}|\nabla_{\gamma}\rho|^{\alpha_{2}}\frac{|\nabla_{x}\rho|^{2}
+|x|^{2\gamma}|\nabla_{y}\rho|^{2}}{\rho^{2}}|f|^{2}dxdy$$
\begin{equation}\label{Hardy4}
=\beta^{2}\int_{\Omega}\rho^{\alpha_{1}}|\nabla_{\gamma}\rho|^{\alpha_{2}}\frac{|\nabla_{\gamma}\rho|^{2}}{\rho^{2}}|f|^{2}dxdy.
\end{equation}
Then by \eqref{Hardy2}, \eqref{Hardy3} and \eqref{Hardy4} we obtain
$$\int_{\Omega}\rho^{\alpha_{1}}|\nabla_{\gamma}\rho|^{\alpha_{2}}|(\nabla_{\gamma}+i\beta\mathcal{A})f|^{2}dz\geq
\left(\left(\frac{Q+\alpha_{1}-2}{2}\right)^{2}+\beta^{2}\right)$$
\begin{equation}\label{Hardy5}
\times\int_{\Omega}\rho^{\alpha_{1}}|\nabla_{\gamma}\rho|^{\alpha_{2}}\frac{|\nabla_{\gamma}\rho|^{2}}{\rho^{2}}|f|^{2}dxdy.
\end{equation}
Then using \eqref{formula}, we observe that \eqref{Hardy5} yields \eqref{Hardy1}. Since the constant in \eqref{Hardy3} is sharp when $0\in\Omega$ by \eqref{Amb1}, then the constant in the obtained inequality is sharp when $0\in\Omega$.
\end{proof}

We obtain the following corollary in $\mathbb{R}^{2+k}$ for the Aharonov-Bohm potential of the type considered in \cite{AL11}.

\begin{cor}\label{another_mag} Let $\Omega\subset\mathbb{R}^{2+k}$ be an open set. Let $(x,y)=(x_{1},x_{2},y)\in \mathbb{R}^{2}\times\mathbb{R}^{k}$. Let $\alpha_{1}, \alpha_{2}, \beta\in\mathbb{R}$ be such that $\alpha_{1}+k(\gamma+1)>0$ and $\alpha_{2}\gamma+2>0$. Then for any real-valued function $f\in C^{\infty}_{0}(\Omega)$ and for the following Aharonov-Bohm type magnetic field
\begin{equation}\label{another_mag2}
\mathcal{\widetilde{A}}=
\left(-\frac{\partial_{x_{2}} \rho}{\rho}, \frac{\partial_{x_{1}} \rho}{\rho},-\frac{|x|^{\gamma}}{\sqrt{2}}
\frac{\nabla_{y} \rho}{\rho}, \frac{|x|^{\gamma}}{\sqrt{2}}
\frac{\nabla_{y} \rho}{\rho}\right),
\end{equation}
we have the following weighted Hardy inequality for the magnetic Baouendi-Grushin operator
$$\int_{\Omega}\rho^{\alpha_{1}}|\widetilde{\nabla_{\gamma}}\rho|^{\alpha_{2}}|(\widetilde{\nabla_{\gamma}}+i\beta\mathcal{\widetilde{A}})f|^{2}dxdy
\geq\left(\left(\frac{\alpha_{1}+k(\gamma+1)}{2}\right)^{2}+\beta^{2}\right)$$
\begin{equation}\label{another_mag1}
\times\int_{\Omega}\rho^{\alpha_{1}}|\widetilde{\nabla_{\gamma}}\rho|^{\alpha_{2}}\frac{|x|^{2\gamma}}{\rho^{2\gamma+2}}|f|^{2}dxdy,
\end{equation}
where \begin{equation}\label{another_mag3}
\widetilde{\nabla_{\gamma}}=\left(\frac{\partial}{\partial x_{1}}, \frac{\partial}{\partial x_{2}}, \frac{|x|^{\gamma}}{\sqrt{2}}\nabla_{y}, \frac{|x|^{\gamma}}{\sqrt{2}}\nabla_{y}\right).
\end{equation} Moreover, if $0\in\Omega$, then the constant $\left(\left(\frac{\alpha_{1}+k(\gamma+1)}{2}\right)^{2}+\beta^{2}\right)$ in \eqref{another_mag1} is sharp.
\end{cor}
\begin{proof}[Proof of Corollary \ref{another_mag}]
In this case $m=2$, then $Q=2+k(1+\gamma)$. Since we have $|\mathcal{\widetilde{A}}f|^{2}=|\mathcal{A}f|^{2}$, $|\widetilde{\nabla_{\gamma}}\rho|=|\nabla_{\gamma}\rho|$ and $|\widetilde{\nabla_{\gamma}}f|^{2}=|\nabla_{\gamma}f|^{2}$, then in the exact same way as in the proof of the Lemma \ref{lem1} one obtains \eqref{another_mag1}.
\end{proof}

As another corollary of Lemma \ref{lem1} we obtain the following uncertainty principle.

\begin{cor}[{\rm Uncertainty type principle}]\label{uncer}
Let $(x,y)=(x_{1},...,x_{m}, y_{1},...,y_{k})\in \mathbb{R}^{m}\times\mathbb{R}^{k}$ with $k,m\geq1$, $k+m=n$. Let $\Omega\subset\Rn$ be an open set. Let $\alpha_{1}, \alpha_{2}, \beta\in\mathbb{R}$ be such that $Q+\alpha_{1}-2>0$ and $m+\alpha_{2}\gamma>0$. Then for any real-valued function $f\in C^{\infty}_{0}(\Omega)$ we have
$$\|\rho^{\frac{\alpha_{1}}{2}}|\nabla_{\gamma}\rho|^{\frac{\alpha_{2}}{2}}(\nabla_{\gamma}+i\beta\mathcal{A})f\|_{L^{2}(\Omega)}
\|f\|_{L^{2}(\Omega)}\geq\left(\left(\frac{Q+\alpha_{1}-2}{2}\right)^{2}+\beta^{2}\right)^{\frac{1}{2}}$$
\begin{equation}\label{uncer1}
\times\int_{\Omega}\rho^{\frac{\alpha_{1}}{2}}|\nabla_{\gamma}\rho|^{\frac{\alpha_{2}}{2}}
\frac{|x|^{\gamma}}{\rho^{\gamma+1}}|f|^{2}dxdy.
\end{equation}
\end{cor}
\begin{proof}[Proof of Corollary \ref{uncer}] By Lemma \ref{lem1} we get
$$\|\rho^{\frac{\alpha_{1}}{2}}|\nabla_{\gamma}\rho|^{\frac{\alpha_{2}}{2}}(\nabla_{\gamma}+i\beta\mathcal{A})f\|_{L^{2}(\Omega)}
\|f\|_{L^{2}(\Omega)}\geq\left(\left(\frac{Q+\alpha_{1}-2}{2}\right)^{2}+\beta^{2}\right)^{\frac{1}{2}}
$$$$\times\left\|\rho^{\frac{\alpha_{1}}{2}}|\nabla_{\gamma}\rho|^{\frac{\alpha_{2}}{2}}
\frac{|x|^{\gamma}}{\rho^{\gamma+1}}f\right\|_{L^{2}(\Omega)}\|f\|_{L^{2}(\Omega)} $$
$$\geq \left(\left(\frac{Q+\alpha_{1}-2}{2}\right)^{2}+\beta^{2}\right)^{\frac{1}{2}}
\int_{\Omega}\rho^{\frac{\alpha_{1}}{2}}|\nabla_{\gamma}\rho|^{\frac{\alpha_{2}}{2}}
\frac{|x|^{\gamma}}{\rho^{\gamma+1}}|f|^{2}dxdy.$$
The proof is complete.
\end{proof}

We now give the main theorem of this section for complex-valued functions $f$.

\begin{thm}\label{thm2}
Let $(x,y)=(x_{1}, x_{2}, y)\in \mathbb{R}^{2}\times\mathbb{R}^{k}$. Let $\alpha_{1},\alpha_{2}, \beta\in\mathbb{R}$ be such that $\alpha_{1}+k(\gamma+1)>0$ and $\alpha_{2}+2\gamma>0$. Recall $\widetilde{\nabla_{\gamma}}$, $\mathcal{\widetilde{A}}$ and $\rho$ defined in \eqref{another_mag3}, \eqref{another_mag2}, \eqref{dist}, respectively.

 Then for any complex-valued function $f\in C^{\infty}_{0}(\mathbb{R}^{2+k}\backslash\{0\})$ we have the following weighted Hardy inequality for the magnetic Baouendi-Grushin operator
\begin{multline}\label{thm2_0}
\int_{\mathbb{R}^{2+k}}\rho^{\alpha_{1}}|\widetilde{\nabla_{\gamma}}\rho|^{\alpha_{2}}|(\widetilde{\nabla_{\gamma}}+i\beta\mathcal{\widetilde{A}})f|
^{2}dxdy
\\ \geq\left(\left(\frac{\alpha_{1}+k(\gamma+1)}{2}\right)^{2}+\beta^{2}\right)
\int_{\mathbb{R}^{2+k}}\rho^{\alpha_{1}}|\widetilde{\nabla_{\gamma}}\rho|
^{\alpha_{2}}\frac{|x|^{2\gamma}}{\rho^{2\gamma+2}}|f|^{2}dxdy
\\ +\int_{\mathbb{R}^{2+k}}\rho^{\alpha_{1}}|\widetilde{\nabla_{\gamma}}\rho|
^{\alpha_{2}}\frac{|f|^{2}-|f_{0}(|x|,y)|^{2}}{|x|^{2}}dxdy,
\end{multline}
with the remainder term, where $f_{0}(|x|,y)=\frac{1}{2\pi}\int_{0}^{2\pi}f(|x|,\phi, y)d\phi$,
$(|x|,\phi)$ is the polar decomposition of $x$. The latter remainder term is non-negative, therefore, we also have the following inequality for any complex-valued function $f\in C^{\infty}_{0}(\mathbb{R}^{2+k}\backslash\{0\})$
\begin{multline}\label{thm2_1}
\int_{\mathbb{R}^{2+k}}\rho^{\alpha_{1}}|\widetilde{\nabla_{\gamma}}\rho|^{\alpha_{2}}|(\widetilde{\nabla_{\gamma}}+i\beta\mathcal{\widetilde{A}})f|
^{2}dxdy
\\
\geq\left(\left(\frac{\alpha_{1}+k(\gamma+1)}{2}\right)^{2}+\beta^{2}\right)
\int_{\mathbb{R}^{2+k}}\rho^{\alpha_{1}}|\widetilde{\nabla_{\gamma}}\rho|
^{\alpha_{2}}\frac{|x|^{2\gamma}}{\rho^{2\gamma+2}}|f|^{2}dxdy,
\end{multline}
and the constant $\left(\frac{\alpha_{1}+k(\gamma+1)}{2}\right)^{2}+\beta^{2}$ is sharp.
\end{thm}
The proof of Theorem \ref{thm2} will be based on the following theorem:
\begin{thm}\label{radial_ineq} Let $(x,y)=(x_{1},...,x_{m}, y_{1},...,y_{k})\in \mathbb{R}^{m}\times\mathbb{R}^{k}$ with $k,m\geq1$, $k+m=n$. Let $\alpha_1,\alpha_2\in\mathbb R$ be such that
$$Q+\alpha_{1}-2>0 \textrm{ and } m+\gamma \alpha_{2}>0.$$
Then we have the following Hardy type inequality for all complex-valued functions $f\in C_{0}^{\infty}(\Rn\backslash\{0\})$,
\begin{multline}\label{radial_ineq1}
\int_{\Rn}\rho^{\alpha_{1}}|\nabla_{\gamma}\rho|^{\alpha_{2}}\left(
\left|\frac{d}{d|x|}f\right|^{2}+|x|^{2\gamma}|\nabla_{y}f|^{2}\right)dxdy
\\ \geq \left(\frac{Q+\alpha_{1}-2}{2}\right)^{2}
\int_{\Rn}\rho^{\alpha_{1}}|\nabla_{\gamma}\rho|^{\alpha_{2}}\frac{|\nabla_{\gamma}\rho|^{2}}{\rho^{2}}|f|^{2}dxdy,
\end{multline}
with sharp constant $\left(\frac{Q+\alpha_{1}-2}{2}\right)^{2}$.
\end{thm}

\begin{rem}\label{rem_radial_der} Theorem \ref{radial_ineq} implies the following inequality
\begin{multline}\label{rem_radial_der1}
\int_{\Rn}\rho^{\alpha_{1}}|\nabla_{\gamma}\rho|^{\alpha_{2}}\left(
\left|\nabla_{x}f\right|^{2}+|x|^{2\gamma}|\nabla_{y}f|^{2}\right)dxdy
\\ \geq \left(\frac{Q+\alpha_{1}-2}{2}\right)^{2}
\int_{\Rn}\rho^{\alpha_{1}}|\nabla_{\gamma}\rho|^{\alpha_{2}}\frac{|\nabla_{\gamma}\rho|^{2}}{\rho^{2}}|f|^{2}dxdy,
\end{multline}
with sharp constant, which gives the result of D'Ambrosio \eqref{Amb1} when $p=2$ and $\Omega=\Rn$. We also mention that inequality \eqref{rem_radial_der1} has been established in \cite{Kombe15} and \cite{ShY12}.
\end{rem}
\begin{proof}[Proof of Theorem \ref{radial_ineq}] We denote $r=|x|$ and $B(r,y)=\rho^{\alpha_{1}}|\nabla_{\gamma}\rho|^{\alpha_{2}}$. Then, using \eqref{dist} and \eqref{formula}, one has
\begin{equation}\label{B(r,y)}
B(r,y)=\rho^{\alpha_{1}}|\nabla_{\gamma}\rho|^{\alpha_{2}}
=r^{\alpha_{2}\gamma}\rho^{\alpha_{1}-\alpha_{2}\gamma}=r^{\alpha_{2}\gamma}(r^{2(1+\gamma)}+(1+\gamma)^{2}
|y|^{2})^{\frac{\alpha_{1}-\alpha_{2}\gamma}{2(1+\gamma)}}.
\end{equation}
Let us first calculate the following
$$\int_{\mathbb{R}^{k}}\int_{0}^{\infty}\left(\left|\left(\partial_{r}+\alpha\frac{\partial_{r}\rho}{\rho}\right)f\right|^{2}+
r^{2\gamma}\left|\left(\nabla_{y}+\alpha\frac{\nabla_{y}\rho}{\rho}\right)f\right|^{2}\right)r^{m-1}B(r,y)drdy$$
$$=\int_{\mathbb{R}^{k}}\int_{0}^{\infty}\left(|\partial_{r}f|^{2}+r^{2\gamma}|\nabla_{y}f|^{2}\right)r^{m-1}B(r,y)drdy$$
$$+\alpha^{2}\int_{\mathbb{R}^{k}}\int_{0}^{\infty}\left(\left|\frac{\partial_{r}\rho}{\rho}\right|^{2}+
r^{2\gamma}\left|\frac{\nabla_{y}\rho}{\rho}\right|^{2}\right)|f|^{2}r^{m-1}B(r,y)drdy$$
$$+2\alpha{\rm Re}\int_{\mathbb{R}^{k}}\int_{0}^{\infty}\frac{\partial_{r}\rho}{\rho}r^{m-1}B(r,y)\overline{\partial_{r}f}\cdot fdrdy$$
$$+2\alpha{\rm Re}\int_{\mathbb{R}^{k}}\int_{0}^{\infty}\frac{\nabla_{y}\rho}{\rho}\cdot\overline{\nabla_{y}f}r^{2\gamma+m-1}B(r,y)fdrdy$$
\begin{equation}\label{radial_ineq2}=:I_{1}+I_{2}+I_{3}+I_{4}.
\end{equation}
Using
$$\frac{\partial_{r}\rho}{\rho}=\frac{r^{2\gamma+1}}{\rho^{2\gamma+2}} \;\;{\rm and} \;\;\frac{\nabla_{y}\rho}{\rho}=\frac{(\gamma+1)y}{\rho^{2\gamma+2}},$$
we calculate
\begin{equation}\label{calcul}
\left|\frac{\partial_{r}\rho}{\rho}\right|^{2}+
r^{2\gamma}\left|\frac{\nabla_{y}\rho}{\rho}\right|^{2}=\frac{r^{4\gamma+2}+r^{2\gamma}(\gamma+1)^{2}|y|^{2}}{\rho^{4\gamma+4}}
=\frac{r^{2\gamma}}{\rho^{2\gamma+2}}=\frac{|\nabla_{\gamma}\rho|^{2}}{\rho^{2}}.
\end{equation}
Thus, we obtain
\begin{equation}\label{radial_ineq3}
I_{2}=\alpha^{2}\int_{-\infty}^{\infty}\int_{0}^{\infty}\frac{|\nabla_{\gamma}\rho|^{2}}{\rho^{2}}|f|^{2}r^{m-1}B(r,y)drdy.
\end{equation}
Now we proceed using integration by parts for $I_{3}$
$$I_{3}=-\alpha\int_{\mathbb{R}^{k}}\int_{0}^{\infty}(2\gamma+m+\gamma \alpha_{2})\rho^{\alpha_{1}-\alpha_{2}\gamma-2\gamma-2}
r^{2\gamma+m-1+\gamma\alpha_{2}}|f|^{2}drdy$$
$$-\alpha\int_{\mathbb{R}^{k}}\int_{0}^{\infty}
(\alpha_{1}-\alpha_{2}\gamma-2\gamma-2)\rho^{\alpha_{1}-\alpha_{2}\gamma-4\gamma-4}r^{4\gamma+m+\gamma \alpha_{2}+1}|f|^{2}drdy.$$
Since $B(r,y)=r^{\alpha_{2}\gamma}\rho^{\alpha_{1}-\alpha_{2}\gamma}$ by \eqref{B(r,y)}, then we have
$$I_{3}=-\alpha\int_{\mathbb{R}^{k}}\int_{0}^{\infty}\left((2\gamma+m+\gamma \alpha_{2})\frac{r^{2\gamma}}{\rho^{2\gamma+2}}+(\alpha_{1}-\alpha_{2}\gamma-2\gamma-2)\frac{r^{4\gamma+2}}{\rho^{4\gamma+4}}
\right)r^{m-1}B(r,y)|f|^{2}drdy$$
$$=-\alpha\int_{\mathbb{R}^{k}}\int_{0}^{\infty}\left(2\gamma+m+\gamma \alpha_{2}+(\alpha_{1}-\alpha_{2}\gamma-2\gamma-2)\frac{r^{2\gamma+2}}{\rho^{2\gamma+2}}
\right)\frac{|\nabla_{\gamma}\rho|^{2}}{\rho^{2}}|f|^{2}r^{m-1}B(r,y)drdy.$$
Similarly, we have for $I_{4}$
$$I_{4}=-\alpha\int_{\mathbb{R}^{k}}\int_{0}^{\infty}{\rm div}_{y}\left(B(r,y)\frac{\nabla_{y}\rho}
{\rho}\right)r^{2\gamma+m-1}|f|^{2}drdy$$
$$=-\alpha(\gamma+1)\int_{\mathbb{R}^{k}}\int_{0}^{\infty}
{\rm div}_{y}\left(\rho^{\alpha_{1}-\alpha_{2}\gamma-2\gamma-2}y\right)r^{\alpha_{2}\gamma+2\gamma+m-1}|f|^{2}drdy$$
$$=-\alpha\int_{\mathbb{R}^{k}}\int_{0}^{\infty}\left((\alpha_{1}-\alpha_{2}\gamma-2\gamma-2)
\rho^{\alpha_{1}-\alpha_{2}\gamma-2\gamma-3}\frac{(\gamma+1)^{2}|y|^{2}}{\rho^{2\gamma+1}}\right)
r^{\alpha_{2}\gamma+2\gamma+m-1}|f|^{2}drdy$$
$$-\alpha\int_{\mathbb{R}^{k}}\int_{0}^{\infty}k(\gamma+1)\rho^{\alpha_{1}-\alpha_{2}\gamma-2\gamma-2}
r^{\alpha_{2}\gamma+2\gamma+m-1}|f|^{2}drdy$$
Since $\frac{r^{2\gamma}}{\rho^{2\gamma+2}}=\frac{|\nabla_{\gamma}\rho|^{2}}{\rho^{2}}$ and $B(r,y)=r^{\alpha_{2}\gamma}
\rho^{\alpha_{1}-\alpha_{2}\gamma}$ by \eqref{calcul} and \eqref{B(r,y)}, respectively, then we obtain
$$I_{4}=-\alpha\int_{\mathbb{R}^{k}}\int_{0}^{\infty}\left((\alpha_{1}-\alpha_{2}\gamma-2\gamma-2)
\frac{(\gamma+1)^{2}|y|^{2}}{\rho^{2\gamma+2}}+k(\gamma+1)\right)\frac{|\nabla_{\gamma}\rho|^{2}}{\rho^{2}}|f|^{2}r^{m-1}B(r,y)drdy.$$
Then, taking into account the definition \eqref{dist}, we get
$$I_{3}+I_{4}=-\alpha\int_{\mathbb{R}^{k}}\int_{0}^{\infty}(\alpha_{1}-\alpha_{2}\gamma-2\gamma-2+2\gamma+m+\gamma \alpha_{2}+k(\gamma+1))\frac{|\nabla_{\gamma}\rho|^{2}}{\rho^{2}}|f|^{2}r^{m-1}B(r,y)drdy,$$
using that $Q=m+(1+\gamma)k$ in \eqref{hom_dim}, one has
\begin{equation}\label{radial_ineq4}
I_{3}+I_{4}=-\alpha\int_{\mathbb{R}^{k}}\int_{0}^{\infty}(Q+\alpha_{1}-2
)\frac{|\nabla_{\gamma}\rho|^{2}}{\rho^{2}}|f|^{2}r^{m-1}B(r,y)drdy.
\end{equation}
Putting \eqref{radial_ineq3} and \eqref{radial_ineq4} in \eqref{radial_ineq2}, we have
$$\int_{\mathbb{R}^{k}}\int_{0}^{\infty}\left(\left|\left(\partial_{r}+\alpha\frac{\partial_{r}\rho}{\rho}\right)f\right|^{2}+
r^{2\gamma}\left|\left(\nabla_{y}+\alpha\frac{\nabla_{y}\rho}{\rho}\right)f\right|^{2}\right)r^{m-1}B(r,y)drdy$$
$$=\int_{\mathbb{R}^{k}}\int_{0}^{\infty}\left(|\partial_{r}f|^{2}+r^{2\gamma}|\nabla_{y}f|^{2}\right)r^{m-1}B(r,y)drdy$$
$$-\left((Q+\alpha_{1}-2
)\alpha-\alpha^{2}\right)\int_{\mathbb{R}^{k}}\int_{0}^{\infty}\frac{|\nabla_{\gamma}\rho|^{2}}{\rho^{2}}|f|^{2}r^{m-1}B(r,y)drdy.$$
By substituting $\alpha=\frac{Q+\alpha_{1}-2}{2}$ and taking into account \eqref{B(r,y)}, we obtain \eqref{radial_ineq1}.

Let us now show the sharpness of the constant in \eqref{radial_ineq1}. Taking into account \eqref{Hardy3} and \eqref{radial_ineq1}, we have
$$\int_{\Rn}\rho^{\alpha_{1}}|\nabla_{\gamma}\rho|^{\alpha_{2}}\left(
\left|\nabla_{x}f\right|^{2}+|x|^{2\gamma}|\nabla_{y}f|^{2}\right)dxdy$$$$\geq\int_{\Rn}\rho^{\alpha_{1}}|\nabla_{\gamma}\rho|^{\alpha_{2}}\left(
\left|\frac{d}{d|x|}f\right|^{2}+|x|^{2\gamma}|\nabla_{y}f|^{2}\right)dxdy$$
$$\geq\left(\frac{Q+\alpha_{1}-2}{2}\right)^{2}\int_{\Rn}\rho^{\alpha_{1}}|\nabla_{\gamma}\rho|^{\alpha_{2}}
\frac{|\nabla_{\gamma}\rho|^{2}}{\rho^{2}}|f|^{2}dxdy,$$
which proves the constant $\left(\frac{Q+\alpha_{1}-2}{2}\right)^{2}$ in \eqref{radial_ineq1} is sharp.
\end{proof}

We are now ready to prove Theorem \ref{thm2}.
\begin{proof}[Proof of Theorem \ref{thm2}] By polar coordinates for $x$-plane $x_{1}=r\cos{\phi}$, $x_{2}=r\sin{\phi}$, we have $r=|x|$ and
$$ \frac{\partial_{x_{1}} \rho}{\rho}=\frac{r^{2\gamma+1}\cos{\phi}}{r^{2(1+\gamma)}+(1+\gamma)^{2}|y|^{2}},
\;\;\frac{\partial_{x_{2}} \rho}{\rho}=\frac{r^{2\gamma+1}\sin{\phi}}{r^{2(1+\gamma)}+(1+\gamma)^{2}|y|^{2}},$$
$$ \frac{\nabla_{y} \rho}{\rho}=\frac{(1+\gamma)y}{r^{2(1+\gamma)}+(1+\gamma)^{2}|y|^{2}}.$$
Thus, we can write
\begin{equation}\label{mag_Ahar1}
\int_{\mathbb{R}^{2+k}}\rho^{\alpha_{1}}|\widetilde{\nabla_{\gamma}}\rho|^{\alpha_{2}}
|(\widetilde{\nabla_{\gamma}}+i\beta\mathcal{\widetilde{A}})f|^{2}dx_{1}dx_{2}dy=I_{1}+I_{2},
\end{equation}
where
$$I_{1}=\int_{\mathbb{R}^{k}}\int_{0}^{2\pi}\int_{0}^{\infty}\left|\left(\cos{\phi}\partial_{r}
-\frac{\sin{\phi}}{r}\partial_{\phi}-i\beta\frac{r^{2\gamma+1}\sin{\phi}}{r^{2(1+\gamma)}+(1+\gamma)^{2}|y|^{2}}\right)
f\right|^{2}rB(r,y)dr d\phi dy$$
$$+\int_{\mathbb{R}^{k}}\int_{0}^{2\pi}\int_{0}^{\infty}\left|\left(\sin{\phi}\partial_{r}
+\frac{\cos{\phi}}{r}\partial_{\phi}+i\beta\frac{r^{2\gamma+1}\cos{\phi}}{r^{2(1+\gamma)}+(1+\gamma)^{2}|y|^{2}}\right)
f\right|^{2}rB(r,y)dr d\phi dy$$
and
$$I_{2}=\frac{1}{2}\int_{\mathbb{R}^{k}}\int_{0}^{2\pi}\int_{0}^{\infty}\left|\left(\nabla_{y}-i\beta
\frac{(1+\gamma)y}{r^{2(1+\gamma)}+(1+\gamma)^{2}|y|^{2}}\right)f\right|^{2}r^{2\gamma+1}B(r,y)dr d\phi dy$$
$$+\frac{1}{2}\int_{\mathbb{R}^{k}}\int_{0}^{2\pi}\int_{0}^{\infty}\left|\left(\nabla_{y}+i\beta
\frac{(1+\gamma)y}{r^{2(1+\gamma)}+(1+\gamma)^{2}|y|^{2}}\right)f\right|^{2}r^{2\gamma+1}B(r,y)dr d\phi dy.$$
By opening brackets we obtain
$$I_{1}=\int_{\mathbb{R}^{k}}\int_{0}^{2\pi}\int_{0}^{\infty}\left(|\partial_{r}f|^{2}+\frac{|\partial_{\phi}f|^{2}}{r^{2}}
+\beta^{2}\left|\frac{r^{2\gamma+1}}{r^{2(1+\gamma)}+(1+\gamma)^{2}|y|^{2}}f\right|^{2}\right)rB(r,y)dr d\phi dy$$
$$-2{\rm Re}\int_{\mathbb{R}^{k}}\int_{0}^{2\pi}\int_{0}^{\infty}\frac{r^{2\gamma}}{r^{2(1+\gamma)}+(1+\gamma)^{2}|y|^{2}}
\partial_{\phi}f\cdot\overline{f}rB(r,y)dr d\phi dy$$
and
$$I_{2}=\int_{\mathbb{R}^{k}}\int_{0}^{2\pi}\int_{0}^{\infty}|\nabla_{y}f|^{2}r^{2\gamma+1} B(r,y)dr d\phi dy
$$$$+
\beta^{2}(1+\gamma)^{2}\int_{\mathbb{R}^{k}}\int_{0}^{2\pi}\int_{0}^{\infty}\frac{|y|^{2}|f|^{2}}
{(r^{2(1+\gamma)}+(1+\gamma)^{2}|y|^{2})^{2}} r^{2\gamma+1}B(r,y)dr d\phi dy.$$
Integrating the second integral in $I_{1}$ by parts we see that
$$2{\rm Re}\int_{\mathbb{R}^{k}}\int_{0}^{2\pi}\int_{0}^{\infty}\frac{r^{2\gamma}}{r^{2(1+\gamma)}+(1+\gamma)^{2}|y|^{2}}
\partial_{\phi}f\cdot\overline{f}rB(r,y)dr d\phi dy$$
$$={\rm Re}\int_{\mathbb{R}^{k}}\int_{0}^{2\pi}\int_{0}^{\infty}\frac{r^{2\gamma}}{r^{2(1+\gamma)}+(1+\gamma)^{2}|y|^{2}}
\partial_{\phi}(|f|^{2})rB(r,y)dr d\phi dy=0.$$
Using the Fourier series for $f$ we can expand
$$f(r,\phi,y)=\sum_{k=-\infty}^{\infty}f_{k}(r,y)e^{ik\phi},$$
and by a direct calculation, we get
\begin{equation*}\begin{aligned}
\frac{1}{r^{2}}\int_{0}^{2\pi}|\partial_{\phi}f|^{2}d\phi & =\frac{2\pi}{r^{2}}\sum_{k}k^{2}|f_{k}(r,y)|^{2}
\\ & =
\frac{2\pi}{r^{2}}\sum_{k\neq0}k^{2}|f_{k}(r,y)|^{2} \\
& \geq\frac{2\pi}{r^{2}}\min_{k\neq0}(k^{2})\sum_{k\neq0}|f_{k}(r,y)|^{2} \\
& =\frac{1}{r^{2}}\int_{0}^{2\pi}|f|^{2}d\phi
-\frac{2\pi}{r^{2}}|f_{0}(r,y)|^{2} \\
& =\frac{1}{r^{2}}\int_{0}^{2\pi}(|f|^{2}-|f_{0}(r,y)|^{2})d\phi,
\end{aligned}
\end{equation*}
where $f_{0}(|z|,y)=\frac{1}{2\pi}\int_{0}^{2\pi}f(|z|,\phi,y)d\phi$.
Now putting the obtained estimates for $I_{1}$ and $I_{2}$ in \eqref{mag_Ahar1} we have
$$\int_{\mathbb{R}^{2+k}}\rho^{\alpha_{1}}|\widetilde{\nabla_{\gamma}}\rho|^{\alpha_{2}}
|(\widetilde{\nabla_{\gamma}}+i\beta\mathcal{\widetilde{A}})f|^{2}dx_{1}dx_{2}dy$$
$$\geq
\int_{\mathbb{R}^{k}}\int_{0}^{2\pi}\int_{0}^{\infty}(|\partial_{r}f|^{2}+r^{2\gamma}|\nabla_{y}f|^{2})rB(r,y)dr d\phi dy$$
$$+\beta^{2}\int_{\mathbb{R}^{k}}\int_{0}^{2\pi}\int_{0}^{\infty}\left(\frac{r^{4\gamma+2}
+(1+\gamma)^{2}|y|^{2}r^{2\gamma}}{(r^{2\gamma+2}+(1+\gamma)^{2}|y|^{2})^{2}}\right)|f|^{2}rB(r,y)dr d\phi dy$$
$$+\int_{\mathbb{R}^{k}}\int_{0}^{2\pi}\int_{0}^{\infty}\frac{|f|^{2}-|f_{0}(r,y)|^{2}}{r^{2}}rB(r,y)dr d\phi dy$$
$$=\int_{\mathbb{R}^{k}}\int_{0}^{2\pi}\int_{0}^{\infty}(|\partial_{r}f|^{2}+r^{2\gamma}|\nabla_{y}f|^{2})rB(r,y)dr d\phi dy$$
$$+\beta^{2}\int_{\mathbb{R}^{k}}\int_{0}^{2\pi}\int_{0}^{\infty}\frac{r^{2\gamma}}
{r^{2\gamma+2}+(1+\gamma)^{2}|y|^{2}}|f|^{2}rB(r,y)dr d\phi dy$$
\begin{equation}\label{mag_Ahar2}
+\int_{\mathbb{R}^{k}}\int_{0}^{2\pi}\int_{0}^{\infty}\frac{|f|^{2}-|f_{0}(r,y)|^{2}}{r^{2}}rB(r,y)dr d\phi dy.
\end{equation}
Since $|\widetilde{\nabla_{\gamma}}\rho|=|\nabla_{\gamma}\rho|$, then putting $m=2$ in \eqref{Hardy3} and taking into account \eqref{B(r,y)} we obtain the following estimate for the first integral in \eqref{mag_Ahar2}
$$\int_{\mathbb{R}^{k}}\int_{0}^{2\pi}\int_{0}^{\infty}(|\partial_{r}f|^{2}+r^{2\gamma}|\nabla_{y}f|^{2})rB(r,y)dr d\phi dy$$
\begin{equation}\label{mag_Ahar3}\geq \left(\frac{\alpha_{1}+k(\gamma+1)}{2}\right)^{2}\int_{\mathbb{R}^{2+k}}\rho^{\alpha_{1}}|\widetilde{\nabla_{\gamma}}\rho|^{\alpha_{2}}
\frac{|\widetilde{\nabla_{\gamma}}\rho|^{2}}{\rho^{2}}|f|^{2}dxdy.
\end{equation}
Let us calculate the second integral in \eqref{mag_Ahar2}
$$\beta^{2}\int_{\mathbb{R}^{k}}\int_{0}^{2\pi}\int_{0}^{\infty}\frac{r^{2\gamma}}
{r^{2\gamma+2}+(1+\gamma)^{2}|y|^{2}}|f|^{2}rB(r,y)dr d\phi dy$$
$$=\beta^{2}\int_{\mathbb{R}^{k}}\int_{0}^{2\pi}\int_{0}^{\infty}\frac{r^{4\gamma+2}\cos^{2}{\phi}}
{(r^{2\gamma+2}+(1+\gamma)^{2}|y|^{2})^{2}}|f|^{2}rB(r,y)dr d\phi dy$$
$$+\beta^{2}\int_{\mathbb{R}^{k}}\int_{0}^{2\pi}\int_{0}^{\infty}\frac{r^{4\gamma+2}\sin^{2}{\phi}}
{(r^{2\gamma+2}+(1+\gamma)^{2}|y|^{2})^{2}}|f|^{2}rB(r,y)dr d\phi dy$$
$$+\beta^{2}\int_{\mathbb{R}^{k}}\int_{0}^{2\pi}\int_{0}^{\infty}\frac{(1+\gamma)^{2}r^{2\gamma}|y|^{2}}
{(r^{2\gamma+2}+(1+\gamma)^{2}|y|^{2})^{2}}|f|^{2}rB(r,y)dr d\phi dy$$
$$=\beta^{2}\int_{\mathbb{R}^{2+k}}\rho^{\alpha_{1}}|\widetilde{\nabla_{\gamma}}\rho|^{\alpha_{2}}
\left|\left(\frac{\partial_{x_{1}} \rho}{\rho}, \frac{\partial_{x_{2}} \rho}{\rho},\frac{|x|^{\gamma}}{\sqrt{2}}
\frac{\nabla_{y} \rho}{\rho}, \frac{|x|^{\gamma}}{\sqrt{2}}
\frac{\nabla_{y} \rho}{\rho}\right)\right|^{2}|f|^{2}dxdy$$
\begin{equation}\label{mag_Ahar4}=
\beta^{2}\int_{\mathbb{R}^{2+k}}\rho^{\alpha_{1}}|\widetilde{\nabla_{\gamma}}\rho|^{\alpha_{2}}
\frac{|\widetilde{\nabla_{\gamma}}\rho|^{2}}{\rho^{2}}|f|^{2}dxdy.
\end{equation}
Putting the estimates \eqref{mag_Ahar3} and \eqref{mag_Ahar4} in \eqref{mag_Ahar2}, we obtain \eqref{thm2_0}. Since we have
$$\int_{\mathbb{R}^{2+k}}\rho^{\alpha_{1}}|\widetilde{\nabla_{\gamma}}\rho|
^{\alpha_{2}}\frac{|f|^{2}-|f_{0}(|x|,y)|^{2}}{|x|^{2}}dxdy\geq0$$
and \eqref{thm2_1} with sharp constant for all real-valued functions by Corollary \ref{another_mag}, then this constant is sharp also in the class of complex-valued functions in \eqref{thm2_1}.
\end{proof}

We record the corresponding uncertainty principle.

\begin{cor}[{\rm Uncertainty type principle}]\label{uncer2} Let $(x,y)=(x_{1}, x_{2}, y)\in \mathbb{R}^{2}\times\mathbb{R}^{k}$.
Let $\alpha_{1},\alpha_{2}, \beta\in\mathbb{R}$ be such that $\alpha_{1}+k(\gamma+1)>0$ and $\alpha_{2}\gamma+2>0$. Then for any complex-valued function $f\in C^{\infty}_{0}(\mathbb{R}^{2+k}\backslash\{0\})$ we have
\begin{multline}\label{uncer21}
\|\rho^{\frac{\alpha_{1}}{2}}|\widetilde{\nabla_{\gamma}}\rho|^{\frac{\alpha_{2}}{2}}(\widetilde{\nabla_{\gamma}}+
i\beta\mathcal{\widetilde{A}})f\|_{L^{2}(\mathbb{R}^{2+k})}
\|f\|_{L^{2}(\mathbb{R}^{2+k})}
\\ \geq\left(\left(\frac{\alpha_{1}+k(\gamma+1)}{2}\right)^{2}+\beta^{2}\right)^{\frac{1}{2}}
\int_{\mathbb{R}^{2+k}}\rho^{\frac{\alpha_{1}}{2}}|\widetilde{\nabla_{\gamma}}\rho|^{\frac{\alpha_{2}}{2}}
\frac{|x|^{\gamma}}{\rho^{\gamma+1}}|f|^{2}dxdy.
\end{multline}
\end{cor}
\begin{proof}[Proof of Corollary \ref{uncer2}] Using \eqref{thm2_1} and in a similar way as in the proof of the Corollary \ref{uncer}, one obtains \eqref{uncer21}.
\end{proof}

\section{Hardy inequalities for
Landau-Hamiltonian} \label{Sec4}

In this section we show the Hardy inequalities for the twisted Laplacian with Landau-Hamiltonian type magnetic field.

We adapt the notation of Section \ref{SEC:introLH}, namely, we have
$\nabla_{\mathcal{L}}$ the gradient operator associated with $\mathcal{L}$:
\begin{equation}\label{twist2-2}
\nabla_{\mathcal{L}}f=(\widetilde{X}_{1}f,...,\widetilde{X}_{n}f, \widetilde{Y}_{1}f,...,\widetilde{Y}_{n}f),
\end{equation}
for the vector fields  $\widetilde{X}_{j}=\partial_{x_{j}}-\frac{1}{2}iy_{j}$ and $\widetilde{Y}_{j}=\partial_{y_{j}}+\frac{1}{2}ix_{j}$.

Let us introduce now the generalised form of the twisted Laplacian
$$\widetilde{\mathcal{L}}=\sum_{j=1}^{n}\left[\left(i\partial_{x_{j}}+\psi(|z|)y_{j}\right)^{2}+
\left(i\partial_{y_{j}}-\psi(|z|)x_{j}\right)^{2}\right],$$
where $\psi(|z|)$ is a radial real-valued differentiable function.
Setting
$$\check{{X}_{j}}=\partial_{x_{j}}-i\psi(|z|)y_{j} \textrm{ and }
\check{Y}_{j}=\partial_{y_{j}}+i\psi(|z|)x_{j},$$
we write
\begin{equation}\label{twist3_1_2}
\widetilde{\nabla_{\mathcal{L}}}f=(\check{X}_{1}f, \ldots, \check{X}_{n}f, \check{Y}_{1}f, \ldots, \check{Y}_{n}f).
\end{equation}
The classical Landau Hamiltonian is a special case of this with the choice of the constant function $\psi(|z|)=\frac12$.

\begin{thm}\label{twist_thm_1} Let $\theta_{1}, \theta_{2}, \theta_{3}, \theta_{4}, a, b \in\mathbb{R}$ with $a,b>0$, $\theta_{1}\neq0$, $\theta_{2} \theta_{3}<0$ and $2\theta_{4}\leq \theta_{2}\theta_{3}$. Let $\psi=\psi(|z|)$ be a radial real-valued function such that $\psi\in L^{2}_{loc}(\mathbb{C}\backslash\{0\})$. Let $\Omega$ be a bounded domain in $\mathbb{C}$ and $R=\underset{z\in\Omega}{\rm sup}\{|z|\}$.
For a function $f$, we will be denoting $f_{0}(|z|):=\frac{1}{2\pi}\int_{0}^{2\pi}f(|z|,\phi)d\phi$.

Then we have the following inequalities with remainders:

(i) the weighted Hardy-Sobolev inequality
\begin{equation}\label{twist3_1}
\int_{\mathbb{C}}\frac{|\widetilde{\nabla_{\mathcal{L}}}f|^{2}}{|z|^{2\theta_{1}}}dz-
\theta_{1}^{2}\int_{\mathbb{C}}\frac{|f|^{2}}{|z|^{2\theta_{1}+2}}dz\geq
\int_{\mathbb{C}}\frac{(\psi(|z|))^{2}}{|z|^{2\theta_{1}-2}}|f|^{2}dz+\int_{\mathbb{C}}\frac{|f|^{2}-|f_{0}(|z|)|^{2}}{|z|^{2\theta_{1}+2}}dz,
\end{equation}
for all complex-valued functions $f\in C_{0}^{\infty}(\mathbb{R}^{2}\backslash\{0\})$;

(ii) the logarithmic Hardy inequality
$$\int_{\mathbb{C}}|\widetilde{\nabla_{\mathcal{L}}}f|^{2}|\log|z||^{2}dz-\frac{1}{4}\int_{\mathbb{C}}|f|^{2}dz\geq
\int_{\mathbb{C}}(\psi(|z|))^{2}|z|^{2}|\log|z||^{2}|f|^{2}dz$$
\begin{equation}\label{twist3_1_1}
+\int_{\mathbb{C}}\frac{|f|^{2}-|f_{0}(|z|)|^{2}}{|z|^{2}}|\log|z||^{2}dz,
\end{equation}
for all complex-valued functions $f\in C_{0}^{\infty}(\mathbb{R}^{2}\backslash\{0\})$;

(iii) the Poincar\'e inequality
$$\int_{\Omega}|\widetilde{\nabla_{\mathcal{L}}}f|^{2}dz-\frac{1}{R^{2}}\int_{\Omega}|f|^{2}dz\geq
\int_{\Omega}(\psi(|z|))^{2}|z|^{2}|f|^{2}dz$$
\begin{equation}\label{twist3_1_1_Po}
+\int_{\Omega}\frac{|f|^{2}-|f_{0}(|z|)|^{2}}{|z|^{2}}dz,
\end{equation}
for all complex-valued functions $f\in \widehat{\mathfrak{L}}_{0}^{1,2}(\Omega)$ satisfying $\frac{d}{d|z|}f\in L^{2}(\Omega)$, where the space $\widehat{\mathfrak{L}}_{0}^{1,2}(\Omega)$ is defined in \eqref{space};

(iv) the Hardy-Sobolev inequality with more general weights
$$\int_{\mathbb{C}}\frac{(a+b|z|^{\theta_{2}})^{\theta_{3}}}{|z|^{2\theta_{4}}}|\widetilde{\nabla_{\mathcal{L}}}f|^{2}dz\geq
\frac{\theta_{2}\theta_{3}-2\theta_{4}}{2}\int_{\mathbb{C}}\frac{(a+b|z|^{\theta_{2}})^{\theta_{3}}}{|z|^{2\theta_{4}+2}}|f|^{2}dz$$
\begin{equation}\label{twist3_1_1_sup}
+\int_{\mathbb{C}}\frac{(\psi(|z|))^{2}(a+b|z|^{\theta_{2}})^{\theta_{3}}}{|z|^{2\theta_{4}-2}}|f|^{2}dz
+\int_{\mathbb{C}}\frac{(a+b|z|^{\theta_{2}})^{\theta_{3}}(|f|^{2}-|f_{0}(|z|)|^{2})}{|z|^{2\theta_{4}+2}}dz,
\end{equation}
for all complex-valued functions $f\in C_{0}^{\infty}(\mathbb{R}^{2}\backslash\{0\})$.
\end{thm}
The proof of Theorem \ref{twist_thm_1} will be based on the following family of weighted Hardy inequalities and Poincar\'{e} type inequality that were obtained in \cite[Theorem 3.4 and Theorem 5.1]{RSY16} and \cite[Theorem 8.1]{RSY17}, where $\mathbb{E}=|x|\mathcal{R}$ is the Euler operator and $\mathcal{R}:=\frac{d}{d|x|}$ is the radial derivative.

\begin{thm}\cite[Theorem 3.4]{RSY16}\label{L_p_weighted_th}
Let $\mathbb{G}$ be a homogeneous group
of homogeneous dimension $Q$ and let $\theta\in \mathbb{R}$.
Then for any complex-valued function $f\in C^{\infty}_{0}(\mathbb{G}\backslash\{0\}),$ $1<p<\infty,$ and any homogeneous quasi-norm $|\cdot|$ on $\mathbb{G}$ for $\theta p \neq Q$ we have
\begin{equation}\label{L_p_weighted}
\left\|\frac{f}{|x|^{\theta}}\right\|_{L^{p}(\mathbb{G})}\leq
\left|\frac{p}{Q-\theta p}\right|\left\|\frac{1}{|x|^{\theta}}\mathbb{E} f\right\|_{L^{p}(\mathbb{G})}.
\end{equation}
If $\theta p\neq Q$ then the constant $\left|\frac{p}{Q-\theta p}\right|$ is sharp.
For $\theta p=Q$ we have
\begin{equation}\label{L_p_weighted_log}
\left\|\frac{f}{|x|^{\frac{Q}{p}}}\right\|_{L^{p}(\mathbb{G})}\leq
p\left\|\frac{\log|x|}{|x|^{\frac{Q}{p}}}\mathbb{E} f\right\|_{L^{p}(\mathbb{G})}
\end{equation}
with sharp constant.
\end{thm}
Let $\Omega \subset \mathbb{G}$ be an open set and let $\widehat{\mathfrak{L}}_{0}^{1,p}(\Omega)$ be the completion of $C^{\infty}_{0}(\Omega\backslash\{0\})$ with respect to
$$
\|f\|_{\widehat{\mathfrak{L}}^{1,p}(\Omega)}=\|f\|_{L^{p}(\Omega)}+\|\mathbb{E}f\|_{L^{p}(\Omega)}, \;\;1<p<\infty.
$$
In the abelian case, when $\mathbb{G}=(\mathbb{R}^{2},+)$ and $p=2$, let us give this definition: Let $\Omega \subset \mathbb{R}^{2}$ be an open set and let $\widehat{\mathfrak{L}}_{0}^{1,2}(\Omega)$ be the completion of $C^{\infty}_{0}(\Omega\backslash\{0\})$ with respect to
\begin{equation}\label{space}
\|f\|_{\widehat{\mathfrak{L}}^{1,2}(\Omega)}=\|f\|_{L^{2}(\Omega)}+\|\mathbb{E} f\|_{L^{2}(\Omega)}.
\end{equation}

\begin{thm}\label{Po}\cite[Theorem 5.1]{RSY16}
Let $\Omega$ be a bounded open subset of $\mathbb{G}$. If $1<p<\infty, \;f\in \widehat{\mathfrak{L}}_{0}^{1,p}(\Omega)$ and $\mathcal{R}f\in L^{p}(\Omega)$, then we have
\begin{equation}\label{Po01}
\|f\|_{L^{p}(\Omega)}\leq  \frac{R p}{Q}\|\mathcal{R}f\|_{L^{p}(\Omega)}=\frac{R p}{Q}\left\|\frac{1}{|x|}\mathbb{E}f\right\|_{L^{p}(\Omega)},
\end{equation}
where $R=\underset{x\in \Omega}{\rm sup}|x|$.
\end{thm}
\begin{thm}\label{1}\cite[Theorem 8.1]{RSY17}
Let $\mathbb{G}$ be a homogeneous group
of homogeneous dimension $Q$. Let $a,b>0$, $\theta_{2} \theta_{3}<0$ and $p\theta_{4}-\theta_{2}\theta_{3}\leq Q-p$. Then for any complex-valued function $f\in C^{\infty}_{0}(\mathbb{G}\backslash\{0\}),$ $1<p<\infty,$ and any homogeneous quasi-norm $|\cdot|$ on $\mathbb{G}$, we have
\begin{equation}\label{Lpweighted2}
\frac{Q-p\theta_{4}+\theta_{2}\theta_{3}-p}{p}
\left\|\frac{(a+b|x|^{\theta_{2}})^{\frac{\theta_{3}}{p}}}{|x|^{\theta_{4}+1}}f\right\|_{L^{p}(\mathbb{G})}
\leq\left\|\frac{(a+b|x|^{\theta_{2}})^{\frac{\theta_{3}}{p}}}{|x|^{\theta_{4}}}\mathcal{R}f\right\|_{L^{p}(\mathbb{G})}.
\end{equation}
If $Q\neq p\theta_{4}+p-\theta_{2}\theta_{3}$, then the constant $\frac{Q-p\theta_{4}+\theta_{2}\theta_{3}-p}{p}$ is sharp.
\end{thm}

We briefly recall their proof for the convenience of the reader but also since these will be useful in our argument.

\begin{proof}[Proof of Theorem \ref{L_p_weighted_th}]
Integrating by parts gives for $\theta p \neq Q$
\begin{multline*}
\int_{\mathbb{G}}\frac{|f(x)|^{p}}{|x|^{\theta p}}dx=\int_{0}^{\infty}\int_{\wp}|f(ry)|^{p}r^{Q-1-\theta p}d\sigma(y)dr\\
=-\frac{p}{Q-\theta p}\int_{0}^{\infty} r^{Q-\theta p} {\rm Re} \int_{\wp}|f(ry)|^{p-2} f(ry) \overline{\frac{df(ry)}{dr}}d\sigma(y)dr\\
\leq \left|\frac{p}{Q-\theta p}\right|\int_{\mathbb{G}}\frac{|\mathbb{E}f(x)||f(x)|^{p-1}}{|x|^{\theta p}}dx=
\left|\frac{p}{Q-\theta p}\right|\int_{\mathbb{G}}\frac{|\mathbb{E}f(x)||f(x)|^{p-1}}{|x|^{\theta+\theta (p-1)}}dx.
\end{multline*}
Using H\"{o}lder's inequality, we obtain
$$
\int_{\mathbb{G}}\frac{|f(x)|^{p}}{|x|^{\theta p}}dx\leq \left|\frac{p}{Q-\theta p}\right|\left(\int_{\mathbb{G}}\frac{|\mathbb{E}f(x)|^{p}}{|x|^{\theta p}}dx\right)
^{\frac{1}{p}}\left(\int_{\mathbb{G}}\frac{|f(x)|^{p}}{|x|^{\theta p}}dx\right)^{\frac{p-1}{p}},
$$
which implies \eqref{L_p_weighted}.

In order to show the sharpness of the constant, let us check the equality condition in above H\"older's inequality.
We consider the function
\begin{equation}\label{2}
g_{1}(x)=\frac{1}{|x|^{C}},
\end{equation}
where $C\in\mathbb{R}, C\neq 0$ and $\theta p\neq Q$. Then, a direct calculation implies
\begin{equation}\label{Holder_eq1}
\left|\frac{1}{C}\right|^{p}\left(\frac{|\mathbb{E}g_{1}(x)|}{|x|^{\theta }}\right)^{p}=\left(\frac{|g_{1}(x)|^{p-1}}
{|x|^{\theta (p-1)}}\right)^{\frac{p}{p-1}},
\end{equation}
which satisfies the equality condition in H\"older's inequality.
Thus, the constant $\left|\frac{p}{Q-\theta p}\right|$ is sharp in \eqref{L_p_weighted}.

Now we show \eqref{L_p_weighted_log}. Taking into account $\wp:=\{x\in \mathbb{G}:\,|x|=1\}$ and applying the polar decomposition on homogeneous groups $\mathbb{G}$, then using integration by parts, one calculates
\begin{multline*}
\int_{\mathbb{G}}\frac{|f(x)|^{p}}{|x|^{Q}}dx=\int_{0}^{\infty}\int_{\wp}|f(ry)|^{p}r^{Q-1-Q}d\sigma(y)dr\\
=-p\int_{0}^{\infty} \log r {\rm Re} \int_{\wp}|f(ry)|^{p-2} f(ry) \overline{\frac{df(ry)}{dr}}d\sigma(y)dr\\
\leq p \int_{\mathbb{G}}\frac{|\mathbb{E}f(x)||f(x)|^{p-1}}{|x|^{Q}}|\log|x||dx=
p\int_{\mathbb{G}}\frac{|\mathbb{E}f(x)|\log|x|||}{|x|^{\frac{Q}{p}}}\frac{|f(x)|^{p-1}}{|x|^{\frac{Q(p-1)}{p}}}dx.
\end{multline*}
Using again the H\"{o}lder's inequality, we get
$$
\int_{\mathbb{G}}\frac{|f(x)|^{p}}{|x|^{Q}}dx\leq p\left(\int_{\mathbb{G}}\frac{|\mathbb{E}f(x)|^{p}|\log|x||^{p}}{|x|^{Q}}dx\right)
^{\frac{1}{p}}\left(\int_{\mathbb{G}}\frac{|f(x)|^{p}}{|x|^{Q}}dx\right)^{\frac{p-1}{p}},
$$
which gives \eqref{L_p_weighted_log}.

As in the case of \eqref{L_p_weighted}, we consider the following function to show the sharpness of the constant in \eqref{L_p_weighted_log}
$$g_{2}(x)=(\log|x|)^{C},$$
where $C\in\mathbb{R}$ and $C\neq 0$.
Then, one has
\begin{equation}\label{Holder_eq2}
\left|\frac{1}{C}\right|^{p}\left(\frac{|\mathbb{E}g_{2}(x)||\log|x||}{|x|^{\frac{Q}{p}}}\right)^{p}=\left(\frac{|g_{2}(x)|^{p-1}}
{|x|^{\frac{Q (p-1)}{p}}}\right)^{\frac{p}{p-1}},
\end{equation}
which satisfies the equality condition in H\"older's inequality.
\end{proof}

Before proving Theorem \ref{Po}, we first show the following proposition.
\begin{prop}\label{Poprop}\cite[Proposition 5.2]{RSY16}
Let $\Omega \subset \mathbb{G}$ be an open set. If $1<p<\infty, \;f\in \widehat{\mathfrak{L}}_{0}^{1,p}(\Omega)$ and $\mathbb{E}f\in L^{p}(\Omega)$, then we have
\begin{equation} \label{Po1}
\|f\|_{L^{p}(\Omega)}\leq \frac{p}{Q}\|\mathbb{E}f\|_{L^{p}(\Omega)}.
\end{equation}
\end{prop}
\begin{proof}[Proof of Proposition \ref{Poprop}]
Let us consider the function $\zeta:\mathbb{R}\rightarrow\mathbb{R}$, which is an even and smooth function, satisfying
\begin{itemize}
\item
$0\leq\zeta\leq1,$
\item
$\zeta(r)=1 \;\;{\rm if}\;\; |r|\leq1,$
\item
$\zeta(r)=0 \;\;{\rm if}\;\; |r|\geq2.$
\end{itemize}
We set $\zeta_{\lambda}(x):=\zeta(\lambda|x|)$ for $\lambda>0$.
We have the inequality \eqref{Po1} for $f\in C_{0}^{\infty}(\mathbb{G}\backslash\{0\})$ by \eqref{L_p_weighted} when $\theta=0$. There is $\{f_{\ell}\}_{\ell=1}^{\infty} \in C_{0}^{\infty}(\Omega\backslash\{0\})$ such that $f_{\ell}\rightarrow f$ in $\widehat{\mathfrak{L}}_{0}^{1,p}(\Omega)$ as $\ell\rightarrow\infty$. Let $\lambda>0$. From \eqref{L_p_weighted} when $\theta=0$ one gets
$$\|\zeta_{\lambda}f_{\ell}\|_{L^{p}(\Omega)}\leq\frac{p}{Q}\left(\|(\mathbb{E}\zeta_{\lambda})f_{\ell}\|_{L^{p}(\Omega)}+
\|\zeta_{\lambda}(\mathbb{E}f_{\ell})\|_{L^{p}(\Omega)}\right)$$
for any $\ell\geq1$. Then, we can immediately see that
$$\lim_{\ell\rightarrow\infty}\zeta_{\lambda}f_{\ell}=\zeta_{\lambda}f,$$
$$\lim_{\ell\rightarrow\infty}(\mathbb{E}\zeta_{\lambda})f_{\ell}=(\mathbb{E}\zeta_{\lambda})f,$$
$$\lim_{\ell\rightarrow\infty}\zeta_{\lambda}(\mathbb{E}f_{\ell})=\zeta_{\lambda}(\mathbb{E}f)$$
in $L^{p}(\Omega)$. By these properties we obtain
$$\|\zeta_{\lambda}f\|_{L^{p}(\Omega)}\leq\frac{p}{Q}\left\{\|(\mathbb{E}\zeta_{\lambda})f\|_{L^{p}(\Omega)}+
\|\zeta_{\lambda}(\mathbb{E}f)\|_{L^{p}(\Omega)}\right\}.$$
Since
$$|(\mathbb{E}\zeta_{\lambda})(x)|\leq \begin{cases}
\sup|\mathbb{E}\zeta|, \;\;{\rm if}\;\; \lambda^{-1}<|x|<2\lambda^{-1};\\
0, \;\; {\rm otherwise},
\end{cases}$$
one obtains \eqref{Po1} in the limit as $\lambda\rightarrow 0$.
\end{proof}
\begin{proof}[Proof of Theorem \ref{Po}] Since $R=\underset{x\in \Omega}{\rm sup}|x|$ by Proposition \ref{Poprop} one has
\begin{equation*}\|f\|_{L^{p}(\Omega)}\leq \frac{p}{Q}\|\mathbb{E}f\|_{L^{p}(\Omega)}\\ \leq
\frac{R p}{Q}\|\mathcal{R}f\|_{L^{p}(\Omega)}=\frac{R p}{Q}\left\|\frac{1}{|x|}\mathbb{E}f\right\|_{L^{p}(\Omega)},
\end{equation*}
which implies \eqref{Po01}.
\end{proof}
\begin{proof}[Proof of Theorem \ref{1}] Since for $Q=p\theta_{4}+p-\theta_{2}\theta_{3}$ there is nothing
to prove, let us only consider the case $Q\neq p\theta_{4}+p-\theta_{2}\theta_{3}$. Using polar coordinates $(r,y)=(|x|, \frac{x}{\mid x\mid})\in (0,\infty)\times\wp$ on $\mathbb{G}$, where $\wp$ is the unit quasi-sphere
\begin{equation}\label{EQ:sphere}
\wp:=\{x\in \mathbb{G}:\,|x|=1\},
\end{equation} and by polar decomposition on $\mathbb{G}$ (see, for example, \cite{FS-Hardy} or \cite{FR}) and using the integration by parts, one calculates
\begin{equation}\label{eqp}
\int_{\mathbb{G}}
\frac{(a+b|x|^{\theta_{2}})^{\theta_{3}}}{|x|^{p\theta_{4}+p}}|f(x)|^{p}dx
=\int_{0}^{\infty}\int_{\wp}
\frac{(a+br^{\theta_{2}})^{\theta_{3}}}{r^{p\theta_{4}+p}}|f(ry)|^{p} r^{Q-1}d\sigma(y)dr.
\end{equation}
Since $\theta_{2} \theta_{3}<0$ and $p\theta_{4}-\theta_{2}\theta_{3}<Q-p$ we get
$$
\int_{\mathbb{G}}
\frac{(a+b|x|^{\theta_{2}})^{\theta_{3}}}{|x|^{p\theta_{4}+p}}|f(x)|^{p}dx$$
$$\leq\int_{0}^{\infty}\int_{\wp}
(a+br^{\theta_{2}})^{\theta_{3}}r^{Q-1-p\theta_{4}-p}\left(\frac{b r^{\theta_{2}}}{a+br^{\theta_{2}}}+\frac{a}{a+br^{\theta_{2}}}
\cdot\frac{Q-p\theta_{4}-p}{Q-p\theta_{4}-p+\theta_{2}\theta_{3}}\right)|f(ry)|^{p} d\sigma(y)dr
$$
$$=\int_{0}^{\infty}\int_{\wp}
\frac{(a+br^{\theta_{2}})^{\theta_{3}}r^{Q-1-p\theta_{4}-p}}{Q-p\theta_{4}-p+\theta_{2}\theta_{3}}
\left(\frac{\theta_{2} \theta_{3} b r^{\theta_{2}}}{a+br^{\theta_{2}}}+Q-p\theta_{4}-p\right)|f(ry)|^{p} d\sigma(y)dr
$$
$$=\int_{0}^{\infty}\int_{\wp}
\frac{d}{dr}\left(\frac{(a+br^{\theta_{2}})^{\theta_{3}}
	r^{Q-p\theta_{4}-p}}{Q-p\theta_{4}-p+\theta_{2}\theta_{3}}\right)|f(ry)|^{p}
d\sigma(y)dr$$
$$
=-\frac{p}{Q-p\theta_{4}-p+\theta_{2}\theta_{3}}\int_{0}^{\infty}(a+br^{\theta_{2}})^{\theta_{3}}r^{Q-p\theta_{4}-p}  \,{\rm Re}\int_{\wp}
|f(ry)|^{p-2} f(ry) \overline{\frac{df(ry)}{dr}}d\sigma(y)dr
$$
$$\leq \left|\frac{p}{Q-p\theta_{4}-p+\theta_{2}\theta_{3}}\right|\int_{\mathbb{G}}\frac{(a+b|x|^{\theta_{2}})^{\theta_{3}}|\mathcal{R}f(x)||f(x)|^{p-1}}{|x|^{p\theta_{4}
+p-1}}dx
$$
$$=\frac{p}{Q-p\theta_{4}-p+\theta_{2}\theta_{3}}
\int_{\mathbb{G}}\frac{(a+b|x|^{\theta_{2}})^
{\frac{\theta_{3}(p-1)}{p}}|f(x)|^{p-1}}{|x|^{(\theta_{4}+1)(p-1)}}
\frac{(a+b|x|^{\theta_{2}})^
{\frac{\theta_{3}}{p}}}{|x|^{\theta_{4}}}|\mathcal{R}f(x)|dx.$$
Here the H\"{o}lder's inequality gives
$$
\int_{\mathbb{G}}
\frac{(a+b|x|^{\theta_{2}})^{\theta_{3}}}{|x|^{p\theta_{4}+p}}|f(x)|^{p}dx$$
$$\leq\frac{p}{Q-p\theta_{4}-p+\theta_{2}\theta_{3}}\left(\int_{\mathbb{G}}\frac{(a+b|x|^{\theta_{2}})^{\theta_{3}}}{|x|^{p\theta_{4}
+p}}|f(x)|^{p}dx\right)^\frac{p-1}{p}
\left(\int_{\mathbb{G}}\frac{(a+b|x|^{\theta_{2}})^{\theta_{3}}}{|x|^{p\theta_{4}}}|\mathcal{R}f(x)|^{p}dx\right)^\frac{1}{p},
$$
which implies \eqref{Lpweighted2}.

In order to show the sharpness of the constant, we check the equality condition in above H\"older's inequality.
We consider the function
$$g_{3}(x)=|x|^{C}, $$
where $C\in\mathbb{R}, C\neq 0$ and $Q\neq p\theta_{4}+p-\theta_{2}\theta_{3}$. Then, a direct calculation gives
\begin{equation}\label{Holder_eq2}
\left|\frac{1}{C}\right|^{p}\left(\frac{(a+b|x|^{\theta_{2}})^
{\frac{\theta_{3}}{p}}|\mathcal{R}g_{3}(x)|}{|x|^{\theta_{4}}}\right)^{p}=\left(\frac{(a+b|x|^{\theta_{2}})^
{\frac{\theta_{3}(p-1)}{p}}|g_{3}(x)|^{p-1}}
{|x|^{(\theta_{4}+1) (p-1)}}\right)^{\frac{p}{p-1}},
\end{equation}
which satisfies the equality condition in H\"older's inequality.
This shows the sharpness of the constant $\frac{Q-p\theta_{4}-p+\theta_{2}\theta_{3}}{p}$ in \eqref{Lpweighted2}.
\end{proof}

We are now ready to prove Theorem \ref{twist_thm_1}.
\begin{proof}[Proof of Theorem \ref{twist_thm_1}] Let $\kappa(|z|)\neq0$ be a radial function. By polar coordinates for the $z$-plane $x=r\cos{\phi}$, $y=r\sin{\phi}$, we have
$$\int_{\mathbb{C}}\frac{|\widetilde{\nabla_{\mathcal{L}}}f|^{2}}{\kappa(|z|)}dz=\int_{\mathbb{C}}\left(\left|\left(i\partial_{x}+
y\psi(|z|)\right)f\right|^{2}+\left|\left(i\partial_{y}-
x\psi(|z|)\right)f\right|^{2}\right)\frac{dz}{\kappa(|z|)}$$
$$=\int_{0}^{\infty}\int_{0}^{2\pi}\left(\left|\left(i\cos{\phi}\partial_{r}-\frac{i\sin{\phi}}{r}\partial_{\phi}
+
\psi(r)r\sin{\phi}\right)f\right|^{2}\right)rd\phi \frac{dr}{\kappa(r)}
$$$$
+\int_{0}^{\infty}\int_{0}^{2\pi}\left(\left|\left(i\sin{\phi}\partial_{r}+\frac{i\cos{\phi}}{r}\partial_{\phi}-
\psi(r)r\cos{\phi}\right)f\right|^{2}\right)rd\phi\frac{dr}{\kappa(r)}$$
$$=\int_{0}^{\infty}\int_{0}^{2\pi}\left(|\partial_{r}f|^{2}+\frac{|\partial_{\phi}f|^{2}}{r^{2}}+\psi(r)^{2}r^{2}
|f|^{2}\right)rd\phi \frac{dr}{\kappa(r)}
-2{\rm Re}\int_{0}^{\infty}\int_{0}^{2\pi}i\psi(r)\partial_{\phi}f\cdot \overline{f}rd\phi \frac{dr}{\kappa(r)}.$$
Integrating the second integral in this equality by parts we see that
$$2{\rm Re}\int_{0}^{\infty}\int_{0}^{2\pi}i\psi(r)\partial_{\phi}f\cdot \overline{f}rd\phi \frac{dr}{\kappa(r)}=
2{\rm Re}\int_{0}^{\infty}\int_{0}^{2\pi}i\psi(r)\partial_{\phi}(|f|^{2})rd\phi \frac{dr}{\kappa(r)}=0.$$
Let us represent $f$ via its Fourier series
$$f(r,\phi)=\sum_{k=-\infty}^{\infty}f_{k}(r)e^{ik\phi}.$$
Then we obtain
$$\frac{1}{r^{2}}\int_{0}^{2\pi}|\partial_{\phi}f|^{2}d\phi=\frac{2\pi}{r^{2}}\sum_{k}k^{2}|f_{k}(r)|^{2}=
\frac{2\pi}{r^{2}}\sum_{k\neq0}k^{2}|f_{k}(r)|^{2}$$
$$\geq\frac{2\pi}{r^{2}}\min_{k\neq0}(k^{2})\sum_{k\neq0}|f_{k}(r)|^{2}=\frac{1}{r^{2}}\int_{0}^{2\pi}|f|^{2}d\phi
-\frac{2\pi}{r^{2}}|f_{0}(r)|^{2}=\frac{1}{r^{2}}\int_{0}^{2\pi}(|f|^{2}-|f_{0}(r)|^{2})d\phi,$$
where $f_{0}(|z|)=\frac{1}{2\pi}\int_{0}^{2\pi}f(|z|,\phi)d\phi$.

Thus, we arrive at
\begin{multline}\label{twist3_2}\int_{\mathbb{C}}\frac{|\widetilde{\nabla_{\mathcal{L}}}f|^{2}}{\kappa(|z|)}dz \\
\geq
\int_{\mathbb{C}}\frac{1}{\kappa(|z|)}\left|\frac{d}{d|z|}f\right|^{2}dz+
\int_{\mathbb{C}}\frac{|z|^{2}(\psi(|z|))^{2}}{\kappa(|z|)}|f|^{2}dz+
\int_{\mathbb{C}}\frac{|f|^{2}-|f_{0}(|z|)|^{2}}{|z|^{2}\kappa(|z|)}dz.
\end{multline}
Putting $\kappa(|z|)=|z|^{2\theta_{1}}$, it follows that
\begin{multline}\label{twist3_2_1}\int_{\mathbb{C}}\frac{|\widetilde{\nabla_{\mathcal{L}}}f|^{2}}{|z|^{2\theta_{1}}}dz \\ \geq
\int_{\mathbb{C}}\frac{1}{|z|^{2\theta_{1}}}\left|\frac{d}{d|z|}f\right|^{2}dz+
\int_{\mathbb{C}}\frac{(\psi(|z|))^{2}}{|z|^{2\theta_{1}-2}}|f|^{2}dz+
\int_{\mathbb{C}}\frac{|f|^{2}-|f_{0}(|z|)|^{2}}{|z|^{2\theta_{1}+2}}dz.
\end{multline}
In the abelian case $\mathbb{G}=(\mathbb{R}^{2},+)$, $p=2$ and $\theta=\theta_{1}+1$ with $\theta_{1}\neq0$, \eqref{L_p_weighted} implies
$$\int_{\mathbb{C}}\frac{1}{|z|^{2\theta_{1}}}\left|\frac{d}{d|z|}f\right|^{2}dz\geq \theta_{1}^{2}\int_{\mathbb{C}}\frac{|f|^{2}}{|z|^{2\theta_{1}+2}}dz.$$
Putting this in \eqref{twist3_2_1}, we obtain \eqref{twist3_1}.

Putting $\kappa(|z|)=|\log|z||^{-2}$ in \eqref{twist3_2}, one has
$$\int_{\mathbb{C}}|\log|z||^{2}|\widetilde{\nabla_{\mathcal{L}}}f|^{2}dz\geq
\int_{\mathbb{C}}|\log|z||^{2}\left|\frac{d}{d|z|}f\right|^{2}dz+
\int_{\mathbb{C}}(\psi(|z|))^{2}|z|^{2}|\log|z||^{2}|f|^{2}dz$$
\begin{equation}\label{twist3_3}
+\int_{\mathbb{C}}\frac{|f|^{2}-|f_{0}(|z|)|^{2}}{|z|^{2}}|\log|z||^{2}dz.
\end{equation}
In the abelian case $\mathbb{G}=(\mathbb{R}^{2},+)$ and $p=2$, \eqref{L_p_weighted_log} implies
$$\int_{\mathbb{C}}|\log|z||^{2}\left|\frac{d}{d|z|}f\right|^{2}dz\geq \frac{1}{4}\int_{\mathbb{C}}|f|^{2}dz.$$
Putting this in \eqref{twist3_3}, we obtain \eqref{twist3_1_1}.

Putting $\kappa(|z|)=1$ in \eqref{twist3_2}, we get
$$\int_{\mathbb{C}}|\widetilde{\nabla_{\mathcal{L}}}f|^{2}dz\geq
\int_{\mathbb{C}}\left|\frac{d}{d|z|}f\right|^{2}dz+
\int_{\mathbb{C}}(\psi(|z|))^{2}|z|^{2}|f|^{2}dz$$
\begin{equation}\label{twist3_3_Po}
+\int_{\mathbb{C}}\frac{|f|^{2}-|f_{0}(|z|)|^{2}}{|z|^{2}}dz.
\end{equation}
In the abelian case, when $\Omega$ is a bounded domain in $(\mathbb{R}^{2},+)$, and $p=2$, \eqref{Po01} implies
$$\int_{\Omega}\left|\frac{d}{d|z|}f\right|^{2}dz\geq \frac{1}{R^{2}}\int_{\Omega}|f|^{2}dz.$$
Putting this in \eqref{twist3_3_Po}, we obtain \eqref{twist3_1_1_Po}.

Putting $\kappa(|z|)=\frac{(a+b|z|^{\theta_{2}})^{-\theta_{3}}}{|z|^{-2\theta_{4}}}$ in \eqref{twist3_2}, we have
$$\int_{\mathbb{C}}\frac{(a+b|z|^{\theta_{2}})^{\theta_{3}}}{|z|^{2\theta_{4}}}|\widetilde{\nabla_{\mathcal{L}}}f|^{2}dz\geq
\int_{\mathbb{C}}\frac{(a+b|z|^{\theta_{2}})^{\theta_{3}}}{|z|^{2\theta_{4}}}\left|\frac{d}{d|z|}f\right|^{2}dz+
\int_{\mathbb{C}}\frac{\psi^{2}(|z|)(a+b|z|^{\theta_{2}})^{\theta_{3}}}{|z|^{2\theta_{4}-2}}|f|^{2}dz$$
\begin{equation}\label{twist3_3_sup}
+\int_{\mathbb{C}}\frac{(a+b|z|^{\theta_{2}})^{\theta_{3}}(|f|^{2}-|f_{0}(|z|)|^{2})}{|z|^{2\theta_{4}+2}}dz.
\end{equation}
Again, in the abelian case $\mathbb{G}=(\mathbb{R}^{2},+)$ and $p=2$, \eqref{Lpweighted2} gives
$$\frac{\theta_{2}\theta_{3}-2\theta_{4}}{2}
\int_{\mathbb{C}}\frac{(a+b|z|^{\theta_{2}})^{\theta_{3}}}{|z|^{2\theta_{4}+2}}|f|^{2}dz
\leq\int_{\mathbb{C}}\frac{(a+b|z|^{\theta_{2}})^{\theta_{3}}}{|z|^{2\theta_{4}}}\left|\frac{d}{d|z|}f\right|^{2}dz,$$
for $a,b>0$, $\theta_{2} \theta_{3}<0$ and $2\theta_{4}\leq \theta_{2}\theta_{3}$.
Putting this in \eqref{twist3_3_sup}, we obtain \eqref{twist3_1_1_sup}.
\end{proof}

Now we give some inequalities for real-valued functions to show the best estimates one can expect. While estimates for real-valued functions have less physical meaning than those for complex-valued functions in questions of the spectral theory, they also find their use in applications to the existence of real (or positive) solutions to some nonlinear equations.

\begin{rem}\label{twist_rem}
By a direct calculation we have
$$\int_{\mathbb{C}^{n}}|\nabla_{\mathcal{L}}f|^{2}dz=\sum_{j=1}^{n}\int_{\mathbb{C}^{n}}\left(\left|\left(i\partial_{x_{j}}+
\frac{y_{j}}{2}\right)f\right|^{2}+\left|\left(i\partial_{y_{j}}-
\frac{x_{j}}{2}\right)f\right|^{2}\right)dz$$
\begin{equation}\label{twist8}=\int_{\mathbb{C}^{n}}|\nabla f|^{2}dz+\int_{\mathbb{C}^{n}}\frac{|z|^{2}}{4}|f|^{2}dz.
\end{equation}
Using the well-known Hardy inequality for $n\geq1$ for any $f\in W_{\mathcal{L}}^{1,2}(\mathbb{C}^{n})$, we obtain
\begin{equation}\label{twist3}
\int_{\mathbb{C}^{n}}|\nabla_{\mathcal{L}}f|^{2}dz\geq\left(n-1\right)^{2}\int_{\mathbb{C}^{n}}\frac{|f|^{2}}{|z|^{2}}dz
+\int_{\mathbb{C}^{n}}\frac{|z|^{2}}{4}|f|^{2}dz,
\end{equation}
for all real-valued functions $f\in W_{\mathcal{L}}^{1,2}(\mathbb{C}^{n})$ and $n\geq1$. Here $W_{\mathcal{L}}^{1,2}$ 
is defined in \eqref{twist_def}.

Let $\Omega$ is a bounded domain in $\mathbb{C}$ with $0\in\Omega$ and $R\geq e\cdot \underset{z\in\Omega}{\rm sup}\{|z|\}$. Then, in the case $n=1$ we can use the critical Hardy inequality with sharp constant for $f\in W_{\mathcal{L}}^{1,2}(\Omega)$ (see for example \cite{AR01}, \cite{AS02}):
\begin{equation}\label{twist4}
\int_{\Omega}|\nabla_{\mathcal{L}}f|^{2}dz\geq \frac{1}{4}
\int_{\Omega}\frac{|f|^{2}}{|z|^{2}\left(\log\frac{R}{|z|}\right)^{2}}dz
+\int_{\Omega}\frac{|z|^{2}}{4}|f|^{2}dz,
\end{equation}
for all real-valued functions $f\in W_{\mathcal{L}}^{1,2}(\Omega)$ and $n=1$.  
Since the constants in the Hardy and critical Hardy inequalities are sharp, then the constants in the obtained inequalities \eqref{twist3} and \eqref{twist4} are sharp.
\end{rem}
\begin{cor}[{\rm Uncertainty type principle}]\label{uncer_twist}
Let $\Omega$ be a bounded domain in $\mathbb{C}$ with $0\in\Omega$ and $R\geq e\cdot \underset{z\in\Omega}{\rm sup}\{|z|\}$. Then we have
\begin{equation}\label{twist5}
\|\nabla_{\mathcal{L}}f\|_{L^{2}(\mathbb{C}^{n})}\|f\|_{L^{2}(\mathbb{C}^{n})}\geq
\int_{\mathbb{C}^{n}}\left(\sqrt{\frac{(n-1)^{2}}{|z|^{2}}+\frac{|z|^{2}}{4}}\right)\left|f\right|^{2}dz
\end{equation}
for $n\geq1$ and all real-valued functions $f\in W_{\mathcal{L}}^{1,2}(\mathbb{C}^{n})$, and
\begin{equation}\label{twist6}
\|\nabla_{\mathcal{L}}f\|_{L^{2}(\Omega)}\|f\|_{L^{2}(\Omega)}\geq
\int_{\Omega}\left(\sqrt{\frac{1}{4|z|^{2}\left(\log\frac{R}{|z|}\right)^{2}}+\frac{|z|^{2}}{4}}\right)\left|f\right|^{2}dz
\end{equation}
for $n=1$ and all real-valued functions $f\in W_{\mathcal{L}}^{1,2}(\Omega)$.
\end{cor}
\begin{proof}[Proof of Corollary \ref{uncer_twist}] By \eqref{twist3} we get for $n\geq1$
$$\|\nabla_{\mathcal{L}}f\|_{L^{2}(\mathbb{C}^{n})}\|f\|_{L^{2}(\mathbb{C}^{n})}\geq
\left\|\left(\sqrt{\frac{(n-1)^{2}}{|z|^{2}}+\frac{|z|^{2}}{4}}\right)f\right\|
_{L^{2}(\mathbb{C}^{n})}\|f\|_{L^{2}(\mathbb{C}^{n})}$$
$$\geq \int_{\mathbb{C}^{n}}\left(\sqrt{\frac{(n-1)^{2}}{|z|^{2}}+\frac{|z|^{2}}{4}}\right)\left|f\right|^{2}dz,$$
which gives \eqref{twist5}.
Similarly, using \eqref{twist4} we obtain \eqref{twist6}.
The proof is complete.
\end{proof}
\begin{rem}
In the case $n=1$ of the Remark \ref{twist_rem}, we also can use the another type of the critical Hardy inequality (see for example Solomyak \cite{Solomyak94}) in \eqref{twist8}:
$$\int_{\mathbb{C}}|\nabla_{\mathcal{L}}f|^{2}dz\geq C
\int_{\mathbb{C}}\frac{|f|^{2}}{|z|^{2}(1+\log^{2}|z|)}dz
+\int_{\mathbb{C}}\frac{|z|^{2}}{4}|f|^{2}dz,$$
where $C$ is a positive constant.
\end{rem}
\begin{rem}\label{Grushin_constant_mag} Let $\mathcal{L}_{G}$ be the magnetic Baouendi-Grushin operator on $\mathbb{C}^{n}$ with the constant magnetic field
$$\mathcal{L}_{G}=\sum_{j=1}^{n}((i\partial_{x_{j}}+\psi_{1,j}(y_{j}))^{2}+(i|x|^{\gamma}\partial_{y_{j}}+\psi_{2,j}(x_{j}))^{2}),$$
where $|x|=\sqrt{|x_{1}|^{2}+\ldots |x_{n}|^{2}}$, $\psi_{1,j},\psi_{2,j}\in L_{loc}^{2}(\mathbb{R}\backslash\{0\})$. Setting $\widehat{X}_{j}=i\partial_{x_{j}}+\psi_{1,j}(y_{j})$ and $\widehat{Y}_{j}=i|x|^{\gamma}\partial_{y_{j}}+\psi_{2,j}(x_{j})$, we write
$$\nabla_{G\mathcal{L}}f=(\widehat{X}_{1}f, \ldots, \widehat{X}_{n}f, \widehat{Y}_{1}f, \ldots, \widehat{Y}_{n}f).$$
Then, a direct calculation gives for all real-valued functions $f\in C_{0}^{\infty}(\mathbb{R}^{2n}\backslash\{0\})$
$$\int_{\mathbb{C}^{n}}\rho^{\alpha_{1}}|\nabla_{\gamma}\rho|^{\alpha_{2}}
|\nabla_{G\mathcal{L}}f|^{2}dz$$
$$=\sum_{j=1}^{n}\int_{\mathbb{C}^{n}}\rho^{\alpha_{1}}|\nabla_{\gamma}\rho|^{\alpha_{2}}
\left(\left|\left(i\partial_{x_{j}}+
\psi_{1,j}(y_{j})\right)f\right|^{2}+\left|\left(i|x|^{\gamma}\partial_{y_{j}}+
\psi_{2,j}(x_{j})\right)f\right|^{2}\right)dz$$
$$=\int_{\mathbb{C}^{n}}\rho^{\alpha_{1}}|\nabla_{\gamma}\rho|^{\alpha_{2}}|\nabla_{\gamma}f|^{2}dz+
\sum_{j=1}^{n}\int_{\mathbb{C}^{n}}\rho^{\alpha_{1}}|\nabla_{\gamma}\rho|^{\alpha_{2}}(|\psi_{2,j}(x_{j})|^{2}
+|\psi_{1,j}(y_{j})|^{2})|f|^{2}dz.$$
Then putting $m=k=n$ in \eqref{Hardy3}, and hence $Q=n(2+\gamma)$, we obtain the following Hardy inequality for magnetic Baouendi-Grushin operator on $\mathbb{C}^{n}$ with the constant magnetic field and for any real-valued function $f\in C_{0}^{\infty}(\mathbb{R}^{2n}\backslash\{0\})$ with sharp constant
$$\int_{\mathbb{C}^{n}}\rho^{\alpha_{1}}|\nabla_{\gamma}\rho|^{\alpha_{2}}
|\nabla_{G\mathcal{L}}f|^{2}dz$$
$$\geq\left(\frac{n(2+\gamma)+\alpha_{1}-2}{2}\right)\int_{\mathbb{C}^{n}}\rho^{\alpha_{1}}|\nabla_{\gamma}\rho|^{\alpha_{2}}
\frac{|x|^{2\gamma}}{\rho^{2\gamma+2}}|f|^{2}dz$$$$+
\sum_{j=1}^{n}\int_{\mathbb{C}^{n}}\rho^{\alpha_{1}}|\nabla_{\gamma}\rho|^{\alpha_{2}}
(|\psi_{2,j}(x_{j})|^{2}+|\psi_{1,j}(y_{j})|^{2})|f|^{2}dz,$$
where $n(2+\gamma)+\alpha_{1}-2>0$ and $n+\alpha_{2}\gamma>0$.
\end{rem}

\end{document}